\documentclass[11pt,twoside,reqno]{amsart}

\usepackage{amssymb}
\usepackage{xcolor}
\usepackage[final]{hyperref}

\hypersetup{unicode= false, colorlinks=true, linkcolor=blue,
anchorcolor=blue, citecolor=green, filecolor=red, menucolor=blue,
pagecolor=blue, urlcolor=blue}

\def\ti{\tilde}

\def\to{\rightarrow}

\def\pa{\partial}

\def \To{\underset
\sphericalangle\longrightarrow}
\def \nlhat {\overset{\ \curlywedge}}

\def\R{{\mathbb R}}

\def\C{{\mathbb{C}}}
\def\N{{\mathbb{N}}}

\def\DD{{\mathcal{D}}}

\def\EE{{\mathcal E}}

\def\e{\varepsilon}
\def\d{\delta}
\def\DD{\Delta}
\def\L{\Lambda}
\def\l{\lambda}
\def\g{\gamma}
\def\G{\Gamma}
\def\a{\alpha}
\def\b{\beta}

\def\lan{\lambda_n}

\def\S{\text{\rm Sinc}}

\theoremstyle{plain}
\newtheorem{lemma}{Lemma}
\newtheorem{theorem}{Theorem}
\newtheorem{corollary}{Corollary}

\newtheorem{remark}{Remark}

\numberwithin{equation}{section}

\author{A.~Poltoratski}
\address{University of Wisconsin\\ Department of Mathematics\\ Van Vleck Hall\\
480 Lincoln Drive\\
Madison, WI  53706\\ USA }
\email{poltoratski@wisc.edu}
\thanks{The author was partially supported by
NSF Grant DMS-2244801.}

\begin{document}

\begin{abstract} We prove pointwise convergence for the scattering data of a Dirac system of differential equations. Equivalently, we
prove an analog of Carleson's theorem on almost everywhere convergence of Fourier series for a version of the non-linear Fourier transform. Our proofs are based
 on the study of resonances of Dirac systems. 
\end{abstract}

\title{Pointwise convergence of  the non-linear Fourier transform}


\maketitle

\section*{Introduction}\label{secInt}

In this paper we study pointwise convergence of the scattering data for a Dirac system of differential equations. Scattering transforms
play an important role in the study of various differential operators and related problems. Extensive evidence
generated in this area during  the last several decades suggests that scattering
can be viewed as a non-linear version of the classical Fourier transform, see for instance \cite{AKNS, ZS, T, TT}.

These connections lead to natural problems of establishing versions of the classical results of Fourier analysis in the non-linear settings of scattering.
Such problems have been appearing in various forms for most of the last century and remain an object of active research today, see for instance \cite{S} for further references. As an example one can look at the non-linear version of Parseval's identity \eqref{eqParseval}, which can be traced as far back as the work of Verblunski in the 1930s, and a non-linear analog
of Hausdorff-Young inequality, which appears in more recent work of Christ and Kiselev  \cite{CK, Ch}.

One of the fundamental results of classical Fourier analysis is the theorem by L. Carleson (1966, \cite{C}) which says that
the Fourier transform
$$\hat f_T(x)=\int_{-T}^T f(s)e^{-ixs}ds$$
converges to $\hat f(x)=\hat f_\infty(x)$ as $T\to\infty$ at almost every point $x\in\R$ for any $f\in L^2(\R)$.
Answering a question by Luzin from 1915, Carleson's theorem
finished a long and turbulent story of partial results and counterexamples created by some of the most prominent mathematicians of the 20th century.
After more than fifty years the existing proofs are still challenging despite several significant contributions by other analysts, including those
by Fefferman (1973, \cite{F}), and by Lacey and  Thiele (2000, \cite{LT}). The theorem and its proofs opened a variety of directions
for further research, see for instance \cite{D, OSTTW} for some of the recent developments and references.

Our main goal is to prove an analog of this classical result in the scattering setting, i.e., to establish pointwise convergence for a version of the non-linear Fourier transform, see Theorem \ref{main} in Section \ref{NLFT}.

Our result implies in particular that generalized eigenfunctions of Dirac systems on the half-line with real $L^2$-potentials are bounded for almost every real spectral parameter, answering a question by Muscalu, Tao and Thiele \cite{MTT1}.
For $1\leq p<2$ this follows from the work of Christ and Kiselev \cite{CK, CK1}.

Convergence in the d-adic model, along with a statement on the d-adic maximal operator, was established
in \cite{MTT}. Further discussion of these problems in the context of Ablowitz-Kaup-Newell-Segur (AKNS) systems can be found
in the book by Muscalu and Schlag (\cite{MS}, Vol. 2, Chapter 5).

Our proof is independent from the linear proofs and is based on a study of resonances of Dirac systems using the methods of spectral
problems for differential operators and complex function theory. While our tools include the basics of the Krein-de Banges theory and its later developments in \cite{MIF1, MIF2, MP},
we are not using any of the deep results of the theory or any of the recent advances of the non-linear Fourier analysis. The necessary background material is  presented in Sections
\ref{secDS}, \ref{secHB} and \ref{secMIF}.

Convergence results are usually closely related to estimates of the maximal operator, as was the case with the original version of Carleson's theorem. While some of such estimates for the non-linear transform can be extracted from the results
of this paper,  sharpness and full extent of such estimates remain unclear and will be studied elsewhere. Let us only mention
that pointwise convergence trivially implies finiteness of the maximal function at almost every point, which was unknown up to now.

\qquad

The plan of our proof is as follows. First we prove universality-type results showing that near almost every point on $\R$
the reproducing kernels of the de Branges spaces corresponding to  a Szeg\"o weight
resemble standard sinc functions. 
Since $L^2$ potentials in Dirac systems correspond exactly to Szeg\"o spectral measures, this implies corresponding approximations in our settings.

Approximations for the reproducing
kernels do not imply approximations for the Hermite-Biehler functions generating the space per se, but under an additional assumption of existence of a resonance
of the system
near a point $s\in\R$, those functions can be approximated by sines and cosines. 

Next, we consider the non-linear Fourier transform
of the restriction  of the potential function $f$ to an interval $(t_1,t_2)$ during which the resonance moves by a fixed distance near $s$. Our approximate knowledge of the Hermite-Biehler functions makes it possible to find similar explicit approximants for  the non-linear Fourier transform. 

At this step a contradiction is obtained by noticing that the expressions for the non-linear Fourier transform of $f$ on the interval $(t_1,t_2)$, obtained under the assumption of the presence of resonances near $s$ for a large $t$, depend not only on 
the values of $f$  on $(t_1,t_2)$ but also on the position and the direction of the motion of the resonance at the initial point $t_1$, which should not be the case. We conclude that 
 for a.e. point $s\in\R$ the system restricted to $(0,t)$ cannot have a resonance  within the distance of $\asymp 1/t$ from $s$ for large enough $t$. 
 
 To finish the proof, we show that for a.e. $s\in\R$, the absence of resonances within $\asymp 1/t$ from $s$ for large $t$ is equivalent to the convergence of
  the non-linear Fourier transform at $s$.

The contents of the sections:

\begin{itemize}

\item In Section \ref{secDS} we introduce the main object of the paper, a Dirac system on $\R_+$ with real potential.

\item Section \ref{secHB} contains basics of Krein-de Branges theory for Dirac systems.

\item Section \ref{secMIF} contains a definition and a brief discussion of meromorphic inner functions (MIFs).

\item In Section \ref{secDIF} we define families of Dirac inner functions satisfying a Riccati equation. 

\item The definition of the scattering matrix and the non-linear Fourier transform, together with their relations with
the Hermite-Biehler functions are presented in Section \ref{NLFT}.

\item In Section \ref{secUR} we prove universality results characterizing the behavior of reproducing kernels of the de Branges
spaces corresponding to the Dirac system near regular points of the spectral measure on the real line.

\item Universality results are translated into approximations of the Hermite-Biehler functions by elementary sines and cosines in Section \ref{secE}.

\item Simultaneous approximations for Hermite-Biehler functions corresponding to Neumann and Dirichlet initial conditions
are obtained in Section \ref{secJA}.

\item In Section \ref{secPMT} we finish the proof of the main result, Theorem \ref{main}.

\item The appendix contains a technical lemma used in Section \ref{secPMT}.

\end{itemize}

\textbf{Acknowledgments:} The approach used in this paper stems from our joint work with Nikolai Makarov, to whom I am grateful for introducing
me to  spectral and scattering problems. I am thankful to Sergey Denisov and Fedor Nazarov for useful discussions related to this work.
I truly appreciate  the help of Lukas Mauth, Gevorg Mnatsakanyan and Christoph Thiele, who read the early versions of the paper and made a number of valuable corrections.
Initial progress in this work was achieved during my sabbatical stay at the  University of Helsinki in April-June of 2019. I am grateful to the math department and especially to my host Eero Saksman for the inspiring working environment and hospitality.

\section{Real Dirac systems}\label{secDS}

We study one of the basic models of scattering corresponding to the 'real' Dirac system on the right half-line $\R_+$,
\begin{equation} \Omega \dot X =z X - QX,
\label{eqDS}\end{equation}
where $z\in\C$ is a spectral parameter,
$$ \Omega=\begin{pmatrix} 0 & 1 \\ -1 & 0 \end{pmatrix},\textrm{ and  }Q(t)=\begin{pmatrix} 0 & f(t) \\ f(t) & 0 \end{pmatrix} $$
for some real-valued locally summable function $f$. A slightly more general form of the system allows
for a locally summable functions $g$ and $-g$ on the main diagonal of $Q$. The function $f+ig$
is then called the potential of the system.
To simplify our exposition, we keep the potential real, although our methods will work similarly
for the general potential.

We will be most interested in  the scattering
problems corresponding to the case $f\in L^2(\R_+)$. For each value of the spectral parameter $z$ the unknown function $$X(t,z)=\begin{pmatrix} u(t,z) \\ v(t,z) \end{pmatrix}$$
is assumed to be differentiable on $\R_+$  with respect to the time variable $t$ and satisfy an initial condition  $X(0,z)=c\in \R^2$.

A special role will be played by solutions  satisfying the Neumann, $$X(0,z)=\begin{pmatrix} 1 \\ 0 \end{pmatrix},$$
and Dirichlet, $$X(0,z)=\begin{pmatrix} 0 \\ 1 \end{pmatrix},$$ initial conditions.
The matrix function $M$ whose columns are the Neumann and Dirichlet solutions, i.e., the matrix-function which solves
\eqref{eqDS} with the initial condition $$M(0,z)=\begin{pmatrix} 1 & 0 \\ 0 & 1 \end{pmatrix},$$ is called
the fundamental matrix, or the transfer matrix,  of the system.

Any Dirac system can be rewritten in the so-called canonical form and henceforth become a part of the Krein-de Branges theory of
canonical systems, see \cite{dBr, DM, R, Ro}. What follows is a brief outline of the basics of the theory as applicable to the subclass of Dirac systems.

\section{Hermite-Biehler functions and spectral measures}\label{secHB}

If $$X(t,z)=\begin{pmatrix} u(t,z) \\ v(t,z) \end{pmatrix}$$ is a solution of \eqref{eqDS}, with an initial condition $X(0,z)=c\in \R^2$, then
for each fixed $t\in \R_+$
the function $$H(t,z)=u(t,z)-iv(t,z)$$ is an Hermite-Biehler entire function, i.e., an entire function satisfying
$$|H(t,z)|>|H(t,\bar z)|$$
for all $z$ from the upper half-plane $\C_+$. Throughout this paper we will denote by $E(t,z)$ and $\ti E(t,z)$ the functions
corresponding to the Neumann and Dirichlet conditions at $0$ correspondingly.

According to Krein's formula for the exponential type (see for instance \cite{R}, Theorem 11), the functions $E(t,\cdot), \ti E(t,\cdot)$ have
exponential type $t$.

If $$M(t,z)=\begin{pmatrix} A(t,z) & B(t,z) \\ C(t,z) & D(t,z) \end{pmatrix}$$ is the fundamental matrix then
$$E(t,z)=A(t,z)-iC(t,z)\textrm{ and }\ti E(t,z)=B(t,z)-iD(t,z).$$
Here $A,B,C,D$ are real (real-valued on the real line) entire functions
analogous to sine and cosine. This analogy becomes precise in the free case $f\equiv 0$ when $$E(t,z)=e^{-itz}, \ti E(t,z)=-ie^{-itz},$$
$$A(t,z)=D(t,z)=\cos tz\textrm{ and }C(t,z)=-B(t,z)=\sin tz.$$

We will use the standard notation $H^\#(z)$ for the Schwarz reflection of an entire function $H^\#(z)=\bar H(\bar z)$. Using this notation,
$$A=\frac{E+E^\#}2\textrm{ and } C=\frac{E-E^\#}{2i}.$$

It follows from \eqref{eqDS} that
\begin{equation}\det M(t,z)= 1\label{det=1}\end{equation}
for all $t$ and $z$.
Rewritten in terms of $E$ and $\ti E$, this relation becomes
\begin{equation}\det\begin{pmatrix} E &  \ti E \\ E^\# & \ti E^\#
\end{pmatrix}\equiv 2i.\label{eqDet2i}\end{equation}

Associated with every Hermite-Biehler function $E(z)$ one can consider a de Branges space $B(E)$, a Hilbert space of entire functions defined as
$$B(E)=\{F| F\text{ is entire, }F/E\in H^2(\C_+),\ F^\#/E\in H^2(\C_+)\},$$
where $H^2(\C_+)$ denotes the standard Hardy space in the upper half-plane. The Hilbert structure in $B(E) $ is inherited from $H^2$:
$$<F,G>_{B(E)}=\int_{\R}F(x)\bar G(x)\frac{dx}{|E(x)|^2}.$$
Each $B(E)$ consists of functions of exponential type at most that of $E$. In particular, in our settings $B(E(t,\cdot))$ contains functions
of exponential type at most $t$.

With this structure $B(E)$ is a reproducing kernel space: point evaluations are bounded linear functionals on the space and, as follows from  the representation theorem, for each $\l\in\C$ there exists $K(\l,\cdot)\in B(E)$ such that for any $F\in B(E)$,
$$F(\l)=<F,K(\l,\cdot)>_{B(E)}.$$
The function $K(\l,z)$ is called the reproducing kernel for the point $\l$. In the case of the de Branges space $B(E)$, $K(\l,z)$  has the formula
$$K(\l,z)=\frac{1}{2\pi i}\frac{  E(z) E^\#(\bar \l)- E^\#(z)E(\bar \l)}{\bar \l -z }=\frac{1}{\pi }\frac{A(z) C(\bar \l) - C(z)A(\bar \l)}{\bar \l-z},
$$
where $A=(E+E^\#)/2$ and $C=(E^\#-E)/2i$ are real entire functions such that $E=A-iC$.

In our settings each space  in the chain $B(E(t,\cdot))$ (in $B(\ti E(t,\cdot))$), possesses its own set of
reproducing kernels $K(t,\l,z)$ ($\ti K(t,\l,z)$).

In the free case $f\equiv 0$ the Hermit-Biehler functions produced by the system \eqref{eqDS} are the exponential functions $E(t,z)=e^{-itz}
\ (\ti E(t,z)=-ie^{-itz})$ and $B(E(t,z))$ is the standard Paley-Wiener space $PW_t$. The reproducing
kernel of $PW_t$ is the sinc function
\begin{equation}\S(t,\l,z)=\frac 1\pi\frac{\sin \left[t(z-\bar\l)\right]}{z-\bar\l}.\label{eqsinc}\end{equation}

It follows from the definition of $K(\l,z)$ that
\begin{equation}\begin{gathered}||K(\l,\cdot)||_{B(E)}=||K(\l,\cdot)/E||_{H^2}=\sqrt{K(\l,\l)}= \\
=\sup_{F\in B(E), ||F||_{B(E)}\leq 1}|F(\l)|.\label{eq1700}\end{gathered}\end{equation}

We denote by $\Pi$ the Poisson measure on $\R$, $d\Pi(x)=dx/(1+x^2)$.
We call a measure $\mu$ on $\R$ Poisson-finite if
$$\int\frac{d|\mu(x)|}{1+x^2}<\infty.$$

The family of de Branges spaces $B(E(t,z)), t\in\R_+,$ possesses a unique positive Poisson-finite measure $\mu$ on $\R$
such that the embedding $B(E(t,z))\to L^2(\mu)$ is isometric for all $t\in\R_+$.
Similarly, the family $B(\ti E(t,z))$ gives rise to a unique measure $\ti\mu$.
The measures  $\mu$ and $\ti\mu$ are called
 the spectral measures for the Dirac system \eqref{eqDS} corresponding to the Neumann and Dirichlet initial conditions at $0$
 correspondingly. See \cite{dBr}, \cite{R}, \cite{MIF1} or \cite{MP} on the basics of spectral theory for canonical systems
 and \cite{Den} or \cite{Ro} on the reduction of Dirac systems to the canonical case.
 
Let $w(x)$ be the density of the absolutely continuous part of $\mu$, $d\mu_{ac}=w(x)dx$, and
let $\ti w$ be the density for the absolutely continuous part of $\ti\mu$.
For $f\in L^2(\R_+)$ the spectral measures satisfy the Szeg\"o condition
$$\log |w|, \ \log |\ti w|\in L^1(\Pi),$$
see the paper by Denisov \cite{Den} for this and many related results.
In particular, $w,\ti w \neq 0$ a.e. on $\R$.

It is well known that in the case of locally summable potentials, and thus in our case when $f\in L^2(\R_+)$,
the spaces $B( E(t,z))$ and $B(\ti E(t,z))$ are equal to the standard Paley-Wiener spaces $PW_t$ as sets (but with different norms).
Indeed, \eqref{eqE} below implies
$$|E(t,x)|, |\ti E(t,x)|\leq e^{\int_0^t|f(s)|ds}$$
 for $x\in\R$ and the condition $\det M=1$ implies that $E$ and $\ti E$
are bounded from below on $\R$. Therefore the norms in $B(E)$ and $B(\ti E)$ are equivalent to the norm in $L^2(\R)$. Together with the property that
$B(E(t,z))$ and $B(\ti E(t,z))$ consist of functions of exponential type at most $t$, we obtain that they are equal to $PW_t$ as sets.

Hence, for the reproducing kernels $K(t,\l,z)$ of $B( E(t,z))$, \eqref{eq1700} can be rewritten as
$$||K(t,\l,\cdot)||_{B(E(t,z))}=\sup_{f\in PW_t, ||f||_{L^2(\mu)}\leq 1}|f(\l)|
$$
and similarly for the reproducing kernels  $\ti K(t,\l,z)$ of $B(\ti  E(t,z))$.

The  equation \eqref{eqDS} rewritten for the Hermite-Biehler function $$E=E(t,z)=u(t,z)-iv(t,z)$$ becomes
$$\frac \pa {\pa  t} E(t,z)= -(zv(t,z) -fu(t,z))-i (zu(t,z)-fv(t,z)),$$
which yields
\begin{equation} \frac \pa {\pa  t} E(t,z)=-izE(t,z) + f (t)E^\# (t,z).\label{eqE}\end{equation}
The initial condition is $E(0,z)=1$ for $E$ and $\ti E(0,z)=-i$ for $\ti E$. 


We denote by $\arg E$ the continuous branch of the argument of $E$ in
the closed upper half-plane satisfying $\arg E(t,0)=0$ ($\arg \ti E(t,0)=-\pi/2$ for $\ti E$).
If $E$ satisfies \eqref{eqE} then for $|E(t,x)|,\ x\in\R,$ we have
$$\frac \pa {\pa  t} |E(t,x)|=f(t)|E(t,x)|\cos [2\arg E(t,x)],$$
which results in
\begin{equation} |E(t,x)|=|E(t_0,x)|\exp\left[\int_{t_0}^tf(t)\cos [2\arg E(t,x)] dt\right],
\label{eqmodE}\end{equation}
for all $t>t_0\geq 0$.

For the continuous branch of $\arg E(t,z)$ on $\R$, \eqref{eqE} gives
\begin{equation}
	 \frac \pa {\pa  t} \arg E(t,x)=-x - f (t)\sin (2\arg E (t,x)).
\label{eqArgE}
\end{equation}

\section{Meromorphic inner functions}\label{secMIF}

Recall that an inner function in the upper half-plane is a bounded analytic function whose non-tangential boundary values
are unimodular a.e. on $\R$, see for instance \cite{G}. An inner function in $\C_+$ is called a meromorphic inner function (MIF) if it can be
continued meromorphically into the whole complex plane. It can be shown that every MIF has a representation
$$\a e^{iaz} \prod \frac{\bar\lan}{\lan}\frac{z-\lan}{z-\bar\lan},
$$
where $\a$ is a unimodular complex constant, $a$ is a positive number, $\{\lan\}$ is a sequence of points in $\C_+$
tending to infinity as $n\to\infty$ and satisfying
the Blaschke condition
$$\sum \frac {\Im\  \lan}{1+|\lan |^2}<\infty.$$

MIFs appear in spectral problems for differential operators
with compact resolvents, see  \cite{MIF1, MIF2}  as well as problems of Fourier analysis,
see \cite{GAP, Type, CBMS}.

Since every MIF is analytic in a neighborhood of the real line, their boundary values and derivatives
are well defined everywhere on $\R$. We will need the following simple lemma relating their derivatives and zeros.

For a sequence $\L=\{\lan\}\subset \C_+$ satisfying the Blaschke condition we denote by $B_\L$ the corresponding
Blaschke product
$$B_\L= \prod \frac{\bar\lan}{\lan}\frac{z-\lan}{z-\bar\lan}.
$$


\begin{lemma} \label{lemInner} Let $\theta$ be a MIF and let $1>\e>0$. Let $x,\ y\in\R$ be such that
\begin{equation}|\theta '(x)|/|\theta'(y)|>1+\e.\label{eq2}\end{equation}
Then  the
ball $\{|z-x|<3|y-x|/\e\}$
contains at least one zero of $\theta$.

\end{lemma}

\begin{proof}
As a MIF, $\theta$ can be represented as
$$\theta(z)=\a e^{icz}B_\L$$
for some unimodular $\a$, $c\geq 0$ and a Blaschke sequence $\L=\{\lan\}\subset \C_+$.
For a Blaschke factor
$$\b_{\lan}=\frac{\bar \l_n}{\lan}\frac{z-\lan}{z-\bar\l_n},\ \lan=x_n+iy_n,$$
the derivative
	of its argument is
$$\frac {2y_n}{(x-x_n)^2+y_n^2}.$$
For $|\theta'(x)|$, which equals to the derivative of the  argument of $\theta$ at $x$, we have	
	$$|\theta'(x)|= c+\sum_n  \frac {2y_n}{(x-x_n)^2+y_n^2}.$$
The sum on the right hand side is a sum of positive functions and for \eqref{eq2} to hold we need a similar
inequality to be satisfied by at least one of the summands, i.e.,
$$\frac {2y_n}{(x-x_n)^2+y_n^2}\cdot\frac {(y-x_n)^2+y_n^2}{2y_n}=\frac{|y-\lan|^2}{|x-\lan|^2}>1+\e
$$
for at least one $n$. The last inequality holds when $\lan$ is in (the interior of an Apollonian circle with foci $x$ and $y$, which is contained in) the ball from the statement. Indeed,
the previous inequality implies $|y-\lan|/|x-\lan|>1+\frac\e 3$
for $1>\e>0$ and 
$$|y-x|\geq |y-\lan| -|x-\lan|>\left(1+\frac\e 3\right)|x-\lan|-|x-\lan|=\e|x-\lan|/3.$$
\end{proof}

\section{Dirac inner functions}\label{secDIF}

In this section we introduce families of inner functions related to the systems \eqref{eqDS}. We call them
Dirac inner functions. In addition to playing a role in our arguments below, such families seem to present
independent interest and may prove useful in the studies of further properties of the system. Standard formulas
expressing Blaschke products in terms of their zeros establish, in this context, the relation between resonances
 (poles of the inner function) and spectra (level sets on $\R$) of the system.

If $H(z)$ is an Hermite-Biehler entire function then the function $$\theta_H(z)=H^\#(z)/H(z)$$ is a meromorphic inner function in $\C_+$.
Under the restriction that $H$ has bounded type, which is the case for functions related to Dirac systems with locally summable potentials,
$H$ can be uniquely, up to a real constant multiple, recovered from $\theta_H$. For more 
on the inner functions $\theta_H$  see
\cite{MIF1, MIF2, MP}.

Recall that
to each Dirac system we associate two families of Hermite-Biehler functions $E(t,z)$ and $\ti E(t,z)$ corresponding
to Neumann and Dirichlet initial conditions at $t=0$. We will denote the corresponding MIFs by
$$\theta=\theta_E\text{ and }\ti \theta=\theta_{\ti E}.$$
Similarly, families of MIFs can be considered for any self-adjoint initial condition.

Note that the functions $u(t,z)$ and $v(t,z)$ from the solution of the Dirac system
\eqref{eqDS} are real entire functions and therefore satisfy $\bar u(t,\bar z)=u(t,z)$
and $\bar v(t,\bar z)=v(t,z)$.
For $\theta=\theta(t,z)$ we have $E(t,z)=u(t,z)-iv(t,z)$ and
$$E^{\#}(t,z)=\bar E(t,\bar z)=\bar u(t,\bar z)+i\bar v(t,\bar z)=u(t,z)+iv(t,z).$$
Hence,
$$\theta(t,z)=\frac{E^\#(t,z)}{E(t,z)}=\frac{u(t, z)+iv(t, z)}{u(t,z)-iv(t,z)}$$
and
$$\frac \pa {\pa  t}\theta(t,z)
=\frac{(\frac \pa {\pa  t}u+i\frac \pa {\pa  t}v)(u-iv)-(u+iv)(\frac \pa {\pa  t}u-i\frac \pa {\pa  t}v)}{(u-iv)^2}=2i\frac{u\frac d{dt}v-v\frac \pa {\pa  t}u}{(u-iv)^2}=$$
$$2i\frac{zu^2+zv^2- 2fuv}{(u-iv)^2}=2iz\frac{u+iv}{u-iv}-4if\frac{uv}{(u-iv)^2}=
$$$$2iz\frac{u+iv}{u-iv}-f\frac{[(u+iv)-(u-iv)][(u+iv)+(u-iv)]}{(u-iv)^2},$$
which produces a Riccati equation for the family of Dirac inner functions corresponding to the system \eqref{eqDS}:
\begin{equation} \frac \pa {\pa  t}\theta=2iz\theta+f(1-\theta^2).\label{eqRiccati}\end{equation}
This equation together with its derivatives will be used in our study of the behavior of resonances of the system \eqref{eqDS}.

\begin{remark} The Riccati equation \eqref{eqRiccati} and corresponding families of functions  pose some interesting  questions.
In our settings  \eqref{eqRiccati} is considered with Neumann initial condition $\theta(0,z)=1$ for $\theta$ and
Dirichlet $\ti\theta(0,z)=-1$ for $\ti\theta$.
One can however consider other initial conditions. Conditions of the type
$\theta(0,z)=\phi(z)$, where $\phi$ is a bounded analytic function in $\C_+$, $|\phi|<1$, present a natural subclass.
One can show that with such initial conditions the solutions $\theta(t,z)$ will remain analytic in $z$ and satisfy $|\theta(t,z)|<1$ for $z\in\C_+$.
Interpreting the complex values as vectors in $\R^2$ one can see that at the real points where $|\theta(t,x)|=1$ the right-hand side of \eqref{eqRiccati} is orthogonal to $\theta(t,z)$, which implies that $\theta$ stays unimodular. Hence, with an inner initial condition
$\theta(t,z)$ is a family of inner functions. One can also show that $\theta(t,z)$ are MIFs if the initial condition is a MIF.

If $\theta(t,z)$ is a Neumann family of Dirac inner functions, each level set $\{z\ |\ \theta(t,z)=\a\},\ |\a|=1$  represents
the spectrum of the system \eqref{eqDS} restricted to the interval $(0,t)$ with the Neumann  condition at $0$ and
the condition $u(t)\sin\psi-v(t)\cos\psi=0$ at $t$, where $\a=e^{2i\psi}$. Analyzing \eqref{eqRiccati} one may follow the dynamics
of the spectra of the system \eqref{eqDS} as $t\to\infty$, see Remark \ref{remND} below.

\end{remark}


Let $\a(t)$ be a continuous curve in $\C$ such that $\theta(t,\alpha(t))=a$ for some constant $a\in \C$.
Then
$$\frac d{d t}\theta(t,\a(t))=0=\theta_t(t,\a(t))+\theta_z(t,\a(t))\a'(t),
$$
which implies
\begin{equation}\a'(t)=-\frac{2i\a(t) a + f(t)(1-a^2)}{\theta_z(t,\a(t))}\label{eqDynamics}
\end{equation}
(for those $t$ for which the denominator is not $0$).
We will be especially interested in the behavior of the zeros of $\theta$, whose complex conjugates represent
the resonances of the Dirac system \eqref{eqDS}. Let $z(t)$ be a curve in $\C_+$ such that $\theta(t,z(t))=0$ for all $t>0$.
Then \eqref{eqDynamics} becomes
\begin{equation}z'(t)=-\frac{f(t)}{\theta_z(t,z(t))}.
\label{eqDynRes}
\end{equation}

In our context the function $E$ is locally approximated by a sine (see Lemma \ref{sinus} below) and
therefore the zeros of $\theta$ are simple. In this case the derivative in the denominator is non-zero and  the application of \eqref{eqDynRes}
is straight-forward.

\begin{remark}\label{remND}
When $a=1$ the point $\a(t)$ belongs to $\sigma_{NN}(t)$, the spectrum of the restriction of the Dirac system
to the interval $(0,t)$ with Neumann conditions on both ends. The dynamics of an eigenvalue $N(t)$ from $\sigma_{NN}$ is
therefore given by
$$N'(t)=-\frac{2i N(t)}{\theta_z(t,N(t))}.$$
Similarly, for the eigenvalues of the Neumann-Dirichlet spectrum, $D(t)$,
we have $a=-1$ and \eqref{eqDynamics} produces the equation
$$D'(t)=\frac{2i D(t)}{\theta_z(t,D(t))}.$$
(Note that the derivative of any MIF on $\R$ is non-vanishing.)

Since for each fixed $t$, $\theta$ is unimodular on $\R$ and its argument is an increasing function, the $z$-derivative of $\theta$ is always
co-linear with $i\theta$. Since $$\theta(t,N(t))=1\textrm{ and }\theta(t,D(t))=-1,$$ the previous two equations imply that $N'(t)$ and $D'(t)$ are negative for the positive eigenvalues and positive
for the negative ones.
This simple observation  implies the known fact that the points of $\sigma_{NN}$ and
$\sigma_{ND}$ tend to zero monotonously as $t\to \infty$.

When $a$ in \eqref{eqDynamics} is unimodular but not equal to $\pm 1$, the point
$\a(t)$ is an eigenvalue for the Neumann condition at 0 and some other $\R^2$-condition at $t$, different from Neumann or Dirichlet. Notice that in this case
the function $f$ does not disappear from the numerator in \eqref{eqDynamics} and the moving eigenvalue changes direction when it
reaches the point $f(t)(1-a^2)/2i a$ (note that this number is real for real $f$ and unimodular $a$). This interesting dynamics of Dirac eigenvalues, viewed as level sets of Dirac inner functions, deserves a separate discussion which we hope to undertake elsewhere.
\end{remark}



\section{Scattering functions and NLFT}\label{NLFT}

Closely related  to the Hermite-Biehler functions $E(t,z)$ and $\ti E(t,z)$ corresponding to the Dirac system \eqref{eqDS} are the scattering
functions
$$\EE(t,z)=e^{itz}E(t,z)\textrm{ and }\ti\EE(t,z)=e^{itz}\ti E(t,z).$$
Without going into a discussion of the physical meaning of the scattering model, let us recall that $e^{-itz}$
is the Hermite-Biehler function of the free system and note that the functions $\EE(t,z)$ and $\ti\EE(t,z)$ represent the propagation of
a wave (signal) up to time $t$ in the system described by \eqref{eqDS}  and return of the same wave in the free system.

Note that scattering functions $\EE$ and $\ti\EE$ satisfy the equation
$$\frac \pa {\pa  t} \EE(t,z)=iz\EE(t,z)+e^{izt}\frac \pa {\pa  t} E(t,z)=$$
\begin{equation}=f(t)e^{2izt}\EE^\#(t,z).
\label{eqH}\end{equation}

Further, for each $t\geq 0$ define entire functions $a(t,z)$ and $b(t,z)$ as
\begin{equation}\begin{gathered}a(t,z)=\frac {\EE(t,z) +i\ti\EE(t,z)}2=\frac {e^{itz}}2 (E(t,z)+i\ti E(t,z)), \\  b(t,z)=
\frac {\EE(t,z) -i\ti\EE(t,z)}2=\frac {e^{itz}}2 (E(t,z)-i\ti E(t,z)).\label{eqab}\end{gathered}\end{equation}
(Note that our notations are slightly different from those in \cite{TT} where $a$ stands for $a^\#$ in our definitions.)

Using \eqref{eqH}, or \eqref{eqE} for $E$ and $\ti E$, one can show that the matrix
$$G(t,z)=\begin{pmatrix}a^\#(t,z) & b^\#(t,z) \\  b(t,z) &  a(t,z)
\end{pmatrix}
$$
satisfies the differential equation
\begin{equation}
G_t=\begin{pmatrix}0 & e^{-2izt}f(t) \\ e^{2izt}f(t) & 0
\end{pmatrix}G
\label{eqG}\end{equation}
with the initial condition $G(0,z)=I$. One can deduce from the IVP that $\det G\equiv 1$ for all $t$ and $z$, which also follows
from \eqref{eqab} and \eqref{det=1}. 
Since $\det G=|a|^2-|b|^2=1$ on $\R$, $|a|>|b|$ on $\R$.
Moreover, it is well known and not difficult to show that $a$ is outer in $\C_+$.
Since $\det G=|a|^2-|b|^2=1$ on $\R$, $|b/a|<1$ on $\R$. Hence $b/a$ is a bounded analytic function in $\C_+$,
$|b/a|<1$.
As was noticed in \cite{TT}, since $|a|^2=|b|^2+1$,  $a(t,0)>0$ and $a$ is outer in $\C_+$, $a,b$ and $G$ can all be
uniquely recovered from $b/a$.

The following formulas implied by \eqref{eqG} and the initial condition $G(0,z)=I$ will be useful to us down the line:

\begin{equation*}
	b^\#(t,z)=\int_{0}^{t}f(s)e^{-2izs}a(s,z)ds,
\end{equation*}
and

	$$a(t,z)=1+\int_{0}^{t}f(s)e^{2izs}b^\#(s,z)ds= $$
\begin{equation}	
	=1+\int_{0}^{t}f(s)e^{2izs}\left[ \int_{0}^{s}f(r)e^{-2izr}a(r,z)dr\right]ds.\label{eqA}
\end{equation}

It is well known that under the restriction $f\in L^2(\R_+)$ the scattering matrix
$$\overbrace{f}=G(\infty, z)=\lim_{t\to\infty} G(t,z)$$
exists, at least in some sense. In the discrete case discussed in \cite{TT} the convergence is proved with respect to a metric on the unit circle (which replaces
the line in the discrete situation), see Lecture 2. Using the same methods one can show convergence in measure for $\log |a(t,\cdot)|$ with respect
to Lebesgue measure on $\R$.
Normal convergence in the upper half-plane is established in \cite{Den} in the equivalent settings of Krein systems in Chapter 12; for the scattering functions $\EE$ and $\ti\EE$ corresponding to Dirac systems it follows from the relations established in Chapter 14 \cite{Den}. Normal convergence in $\C_+$ for $a$ and $b$ then follows from \eqref{eqab}.

Note that since $\log |a(t,\cdot)|$ is a non-negative function on $\R$, convergence of its
Poisson integral at $z=i$ is equivalent to the convergence of the norms $||\log |a(t,\cdot)|\ ||_{L^1(\Pi)}$.
Together with convergence in measure this implies convergence of $\log |a(t,\cdot)|$ in $L^1(\Pi)$.
Since $|b|^2=|a|^2-1$, $|b|$ converges in measure with respect to $\Pi$.
 For $\log_+ |b|=\max(\log |b|,0)$,
using again the relation $|b|^2=|a|^2-1$, we obtain convergence in $L^1(\Pi)$.
Note that for a family of functions from Smirnov class pointwise convergence in $\C_+$  to a non-zero function from Smirnov class, under a restriction
that the outer part is positive at a fixed point, implies pointwise
convergence in $\C_+$ for their outer and inner parts.
Convergence of the outer parts of $b$ at $i$ together with convergence of $\log_+ |b|$ in $L^1(\Pi)$
implies convergence of the $L^1(\Pi)$-norms of $\log_- |b|=\log |b|-\log_+ |b|$. Together with convergence of $|b|$ in measure, we get convergence of
 of $\log |b|$ in $L^1(\Pi)$. For the inner components of $b$,
one can show that pointwise convergence in $\C_+$ implies convergence in measure on $\R$ with respect to $\Pi$. All in all, we obtain that $a$ and  $b$, and therefore $\EE$ and $\ti \EE$,
converge in measure on $\R$ with respect to $\Pi$.

All of the functions $a(t,z)$, including the limit function $a=a(\infty, z)$ for $t=\infty$, satisfy the non-linear version of Parseval's identity
\begin{equation}||\log |a(t,\cdot)|\ ||_{L^1(\R)}=||f||_{L^2(0,t)}^2,\label{eqParseval}\end{equation}
which was known in various forms for many decades, see \cite{Den, MTT, TT} for proofs and further references.

In this paper we pay special attention to the function $b(t,z)/a(t,z)$ and its limit at infinity. Let us denote
$$\nlhat f(z)=\frac {b(\infty,z)}{a(\infty,z)}.$$
If $f_T$ is the restriction of the potential function $f$ to the interval $(0,T)$ (extended by $0$ outside of the interval) then
$$\nlhat f_T(z)=\frac {b(T,z)}{a(T,z)}.$$
There is abundant evidence  that various versions of the scattering transform, including
 $\overbrace{f}$ and $\nlhat f$, can be viewed as  non-linear analogs of the Fourier transform, see for instance
\cite{MTT, S, T, TT} for a discussion and further references. The transform $f\mapsto \nlhat f$ we are about to study
shares the modulation/shift property and the rescaling property with its linear predecessor. Parseval's identity in
terms of $\nlhat f$ takes the form
$$\frac 12||\log (1-|\nlhat f|^2)\ ||_{L^1(\R)}=||f||_{L^2(\R)}^2.$$
The analogy extends further by the property that the transform of a function $f$ supported on a half-line produces
a function $\nlhat f$ holomorphic in the upper half-plane. As was mentioned before, the pair of functions $(a,b)$,
and therefore the matrix $G$, can be uniquely recovered from $\nlhat f$ since $a$ is an outer function
in $\C_+$ which has absolute value $$|a|=\frac1{\sqrt{1-|\nlhat f|^2}}$$ on $\R$ and is positive at $0$, and $b=a\nlhat f$.

%
%
%
%
%
%
%

In this paper we prove the following analog of Carleson's theorem:

\begin{theorem}\label{main} For every real $f\in L^2(\R_+)$,
$$\overset{\ \curlywedge}f_T(s)\to \overset{\ \curlywedge}f(s)\textrm{ as }T\to\infty$$
 for a.e. $s\in\R$, where $f_T$ denotes the restriction of  $f$ to the interval $(0,T)$.
\end{theorem}

We actually prove a slightly stronger statement that for a.e. $s\in \R$ and any $C>0$,
$$\sup_{|z-s|<C/T}|\nlhat f_T(z)-\nlhat f(s)|\to 0\textrm{ as }T\to\infty,$$
see Section \ref{secPMT}.
Such convergence is established for $|E|,\ |\ti E|$ and $\log |a|$.

\section{Universality-type results}\label{secUR}

In this section we show that if the spectral measure of a chain of de Branges spaces satisfies
the Szeg\"o condition, then near almost every point on the real line the scaling limits of the reproducing kernels
are equal to the sinc functions, the reproducing kernels of the Paley-Wiener space. Similar problems, motivated by universality results in random matrix theory, were previously studied
by Lubinski \cite{L} for spaces of polynomials and by Mitkovski \cite{M} for de Branges spaces.
In particular, an analog of Lemma \ref{supK-S} below was proved in \cite{M} for points $s$ of continuity
of density of a regular spectral measure.

We keep our notations of $\mu$ and $\ti\mu$ for the spectral measures of the Dirac system \eqref{eqDS}
with Neumann and Dirichlet initial conditions correspondingly; $w$ and $\ti w$ denote the densities
of the absolutely continuous parts of the measures:
$$d\mu=wdx+d\mu_{s},\ d\ti\mu=\ti wdx+d\ti\mu_{s}.$$
By $K(t,\l, z)$ we denote the reproducing kernel of the space $B(E(t,z))$
for the point $\l$. All our statements can be similarly proved for the reproducing kernels $\ti K(t,\l, z)$
of $\ti B(E(t,z))$.

We use the notations $||\cdot ||_\mu$ and $||\cdot ||_2$ for the norms
in $L^2(\mu)$ and in $L^2(\R)$. For an absolutely continuous measure $\mu,\ d\mu=w(x)dx$, we use $||\cdot||_w$ in
place of $||\cdot||_\mu$.

For a Poisson-finite measure $\mu$ on $\R$ we denote by $P\mu$ its Poisson extension to the upper half-plane
$$P\mu(x+iy)=\frac 1\pi\int\frac{yd\mu(t)}{(x-t)^2+y^2}\ .$$
For $x\in\R$ let $\Gamma(x)$ be the non-tangential sector
$$\G(x)=\{ z\in\C_+\ |\ |\Re\  (z-x)|<\Im\  z <1  \}.
$$
For any function $\phi(z)$ in $\C_+$ we denote by $M\phi$ its non-tangential maximal function on $\R$:
$$M\phi(x)=\sup_{z\in \G(x)} |\phi(z)|.$$
Thus $MP\mu$ will stand for the maximal function of the Poisson extension of $\mu$.

For $s\in\R$ and $C>0$ we will denote by $Q(s,C)$ the square box centered at $s$:
$$Q(s,C)=\{ |\Re\ (z-s)|\leq C, \ |\Im\  z|\leq C     \}.$$
The proximity of reproducing kernels $K(t,z, \cdot)$ to sinc functions $\S(t,z, \cdot)$, defined in \eqref{eqsinc}, will be studied on
boxes $Q(s,C/t)$ whose size decreases with time.

We start with the following statement.

\begin{lemma}\label{lem1200}  For almost all $s\in \R$ and any $C>0$,

\begin{equation} \sup_{z\in Q(s,C/t)}\left(\frac{w(s)||K(t,z, \cdot)||^2_{\mu}}{||\S(t,z, \cdot)||^2_2}-1\right)= o(1),\label{eq1001}\end{equation}

 as $t\to\infty$.
\end{lemma}

\begin{proof}
Recall that
$$||K(t,z, \cdot)||_{\mu}=\sup_{f\in PW_t,\ ||f||_\mu\leq 1} |f(z)| $$
and
$$||\S(t,z, \cdot)||_2=\sup_{f\in PW_t,\ ||f||_2\leq 1} |f(z)|.$$

The relation we need to establish therefore becomes
\begin{equation}\sqrt{w(s)}\sup_{f\in PW_t,\ ||f||_\mu\leq 1} |f(z)|=(1+o(1))\sup_{f\in PW_t,\ ||f||_2\leq 1} |f(z)|.
\label{1900}\end{equation}
Let us first prove that the left hand side of \eqref{1900} is no greater than the right hand side for some choice of
$o(1)$.

Let $G$ be an outer function in $\C_+$ with $|G|^2=w$. Suppose that the non-tangential limit
$G(s)$ exists at  $s$ and $|G(s)|=\sqrt{w(s)}$. Multiplying $G$ by a unimodular constant, we can choose $G$ so that $G(s)=\sqrt{w(s)}$.

Due to the weak-$(1,1)$ type of the non-tangential maximal operator, the function $R=\sqrt{MP\log  w}$
is locally summable and we assume that $s$ is its Lebesgue point.

In this part we will assume that  $G(s)=\sqrt{w(s)}=1$ (otherwise, since $w(s)\neq0$ for a.e. $s$, we can divide $w$ by $w(s)$).
Notice that since $||f||_\mu\geq ||f||_w$,
$$\sup_{f\in PW_t,\ ||f||_\mu\leq 1} |f(z)|\leq\sup_{f\in PW_t,\ ||f||_w\leq 1} |f(z)|,$$
and it is enough to show that the last supremum is less or equal to the right hand side of \eqref{1900}.


 Put $$G_t= G\left(s+\frac {z-s}t\right).$$
 For each $t>0$ choose a point $z_t\in Q(s,C)$. Then
 $$s+\frac{z_t-s}t\in Q(s,C/t).$$
  Using that $||f||_w=||fG||_2$ and rescaling, we obtain
$$\sup_{f\in PW_t,\ ||f||_w\leq 1} \left|f\left(s+\frac {z-s}t\right)\right|=\sqrt{t}\sup_{f\in PW_1,\ ||fG_t||_2\leq 1} |f(z_t)|.$$
The inequality we need to establish for every choice of $z_t\in Q(s,C)$ becomes
\begin{equation}\sup_{f\in PW_1,\ ||fG_t||_2\leq 1} |f(z_t)|\leq (1+o(1))D_t,\label{1011}\end{equation}
where
$$D_t=
\sup_{f\in PW_1,\ ||f||_2\leq 1} |f(z_t)|.$$

Suppose that $f_n$ is a sequence of functions from $PW_1$ such that 
\begin{equation}||f_nG_{k_n}||_2\leq 1\label{eq2022001}\end{equation}
for some $k_n\to\infty$ but
$$f_n(z_{k_n})>D_{k_n}+\e.$$
Notice that all the points $z_{k_n}$ belong to $Q(s,C)$ and therefore, by choosing
a subsequence if necessary, we can assume that $z_{k_n}\to z_0\in Q(s,C)$. Let
$$D=\sup_{f\in PW_1,\ ||f||_2\leq 1} |f(z_0)|.$$ Then $D_{k_n}\to D$.

Let $g_n=e^{iz}f_nG_{k_n}$. 
Recall that $f_n\in PW_1$ and therefore $e^{iz}f_n\in H^2(\C_+)$. Since $G_n$ is outer and $f_nG_n$ satisfy  \eqref{eq2022001},
all $g_n$ are $H^2(\C_+)$-functions of norm at most 1.
By choosing a subsequence if necessary, we can assume that $g_n$ converge to some $g\in H^2$ weakly in $H^2$ (and therefore pointwise in $\C_+$).

Then
$||g||_2\leq 1$. Notice that since $G(z)\to 1$ as $z\To s$, $G_t\to 1$ as $t\to \infty$
normally in $\C_+$. Therefore, the sequence $e^{iz}f_n$ converges to $g$ normally in $\C_+$.

Similarly, by choosing a subsequence if necessary, we can assume that  $e^{-iz}f_n$ converges normally in $\C_-$ to some analytic function $g_-\in H^2(\C_-)$.

Recall that $s$ is a Lebesgue point of $\sqrt{MP\log w}$. Therefore, for an arbitrary large constant $L$ and every $n$ we can choose
 $c_n,\ L<c_n<2L$, such that $P\log w$ is uniformly bounded  on the union of the segments $x-s=\pm c_n/k_n,\ |y|<1$.
Let us consider a square $R_n$ whose sides lie on the lines $\Im\  z=\pm c_n$ and $\Re\  (z-s)=\pm c_n$.
On the vertical sides of $R_n$ in $\C_+$, $$|f_nG_{k_n}|\leq C/\sqrt {|y|},$$ because $$||e^{iz}f_nG_{k_n}||_{H^2}\leq 1$$ and, by the choice of $c_n$, $|G_{k_n}|>\delta>0$ for large enough $n$.
Hence, on the vertical sides of $R_n$ in $\C_+$, and similarly in $\C_-$,  $|f_n(s\pm c_n +i y)|\leq C/\delta \sqrt {|y|}$.
From normal convergence of $f_n$ in $\C_\pm$ and these estimates, we obtain dominated convergence for Cauchy integrals for points
inside $$R=\{|\Im\  z|<L/2,|\Re\  (z-s)|<L/2\}$$ and conclude that $f_n$ converges uniformly on any compact inside $R$. Since $L$ can be arbitrarily large,
it follows that $f_n$ converge normally in $\C$.

Since a
normal limit of a sequence of entire functions is entire, the function $H$ defined as $e^{-iz}g$ in $\C_+$ and as $e^{iz}g_-$ in $\C_-$ extends to an entire function.
The property that $g$ and $g_-$ belong to $H^2(\C_\pm)$ implies that $H\in PW_1$.
From normal convergence of $f_n$ to $H$ it follows that
$$|H(z_0)|=\lim |f_n(z_{k_n})|\geq \lim D_{k_n}+\e=D+\e.$$
Since $||H||_2=||g||_2\leq 1$
we obtain a contradiction.

To prove that the left hand side of \eqref{1900} is no less than the right hand side, let now $z_t$ be a point
in $Q(s,C/t)$ for each $t$. Notice that $$s_t(x)=|\S(t,z_t, x)|^2/||\S(t,z_t, \cdot)||^2_2$$ is an approximative unity at the point $s$ and therefore
$$\left|\int s_t(x) d\mu(x) -w(s)\right| =o(1)\text{ as } t\to\infty$$
for a.e. $s$. Since $\S(t,z_t,\cdot)\in PW_t$,
$$||K(t,z_t, \cdot)||^2_{\mu}=K(t,z_t,z_t)\geq $$$$
\geq[\S(t,z_t,z_t)/||\S(t,z_t,\cdot)||_\mu]^2=$$$$
=[\S(t,z_t,z_t)]^2/\left(||\S(t,z_t, \cdot)||^2_2\int s_t d\mu\right)  \geq $$$$
\geq [\S(t,z_t,z_t)]^2/\left(||\S(t,z_t, \cdot)||^2_2(w(s)+o(1))\right)=$$$$
=||\S(t,z_t, \cdot)||^2_2/(w(s)+o(1)).
$$

\end{proof}



%
%
%
%
%
%

From the asymptotic proximity of norms we can now pass to the proximity of functions themselves.

\begin{lemma}\label{lemL2}
For a.e. $s\in \R$ and any  $C>0$

\begin{equation} \sup_{z\in Q(s,C/t)}\left|\left|K(t,z, \cdot)-\frac{1}{w(s)}\S(t,z, \cdot)\right|\right|^2_{\mu}=o(t)\label{eq1002}
\end{equation}
as $t\to\infty$.

\end{lemma}

Recall that $w(s)\neq 0 $ at a.e. $s$ and therefore the formula above makes sense for a.e. $s$.
We denote by $<\cdot, \cdot>_\mu$ the inner product in
$L^2(\mu)$.

\begin{proof} Let $z_t\in Q(s,C/t)$.
Using the   notation $s_t$ from the  proof of Lemma~\ref{lem1200},
$$<K(t,z_t, \cdot)-\frac{1}{w(s)}\S(t,z_t, \cdot),K(t,z_t, \cdot)-\frac{1}{w(s)}\S(t,z_t, \cdot)>_\mu= $$$$
=K(t,z_t,z_t)+\frac{||\S(t,z_t, \cdot)||^2_2}{w^2(s)}\int s_t d\mu - \frac{2}{w(s)}\S(t,z_t, z_t)=$$$$
= K(t,z_t,z_t)-\frac{1}{w(s)}\S(t,z_t, z_t)+o(||\S(t,z_t, \cdot)||^2_2).
$$
Since
$$\S(t,z_t, z_t)=||\S(t,z_t, \cdot)||^2_2\asymp t,
$$
the statement follows from Lemma \ref{lem1200}.
\end{proof}

From the $L^2$-approximation of the kernels we now pass to the uniform approximation  near $s$.

If $I$ is an interval on $\R$  and $C>0$ we denote by $CI$ the interval with the same center as $I$
of length $C|I|$.

\begin{lemma}\label{supK-S} For a.e. $s\in \R$ and any  $C>0$,
\begin{equation} \sup_{\l,z\in Q(s,C/t))}\left|K(t,\l, z)-\frac{1}{w(s)}\S(t,\l, z)\right|= o(t)\textrm{ as }t\to\infty.     \end{equation}

\end{lemma}

\begin{proof}
Let $z_t\in Q(s,C/t))$.
Let $G$ again be an outer function satisfying $|G|^2=w$.
Define $\Delta(t,z)$ as
$$\Delta(t,z)=K(t,z_t, z)-\frac{1}{w(s)}\S(t,z_t, z).$$

Then, by Lemma \ref{lemL2},
for $F(t,z)=e^{itz}G(z)\Delta(t,z)$ we have
$$||F(t,\cdot)||^2_{H^2}=o(t).$$
From the strong $L^2$ type of the non-tangential maximal operator $M$ it follows that
$$||MF(t,\cdot)||^2_2=o(t).$$

We will denote by $I_t$ the interval $Q(s,C/t)\cap \R$.
 Consider the set $5I_t\setminus 3I_t$ which is a union of two intervals $J^1_t$ and $J^2_t$.
On two-thirds of each of $J_t^1$ and $J_t^2$,
$$(MF)^2\leq  \frac{3||MF||^2_{2}t}{ C}=o(t^2).$$
We denote by $S_t^1$ and $S_t^2$ the subsets of $J_t^1$ and $J_t^2$ correspondingly where this inequality is satisfied.

Once again we notice that the function $R=\sqrt{MP\log  w}$
is locally summable and  assume that $s$ is its Lebesgue point.

 It follows that for sufficiently large  $t$ there exist points $x_1,x_2$ in
$S_t^1,S_t^2$ correspondingly such that $R(x_k)<2R(s)$ and therefore
$$\max_{z\in \G_{x_k}, \Im\  z\leq 3C/t} |\Delta(t,z)|\leq 2e^{3C} e^{2R^2(s)} MF(x_k) =o(t),$$
where the right hand side does not depend on a particular choice of $z_t\in Q(s,C/t))$.
In particular, the last inequality is satisfied on the part of the boundary of the rhombus
\begin{equation}\C\setminus \cup_{x\not\in (x_1,x_2)} (\Gamma_x \cup \bar\Gamma_x)\label{eq0A}\end{equation}
in $\C_+$.
Since $\Delta$ is a real entire function for each fixed $t$, the last inequality
must also be satisfied on the boundary of the rhombus in $\C_-$, and therefore inside the rhombus. It remains to notice
that the rhombus contains the box $Q(s,C/t))$.
\end{proof}

\section{Back to $E$}\label{secE}

From the estimates of reproducing kernels obtained in the previous section we now obtain estimates
for the Hermite-Biehler functions $E$ and $\ti E$. Recall that $E=A-iC$ and $\ti E=B-iD$ for
$$A(t,z)=u(t,z),\ C(t,z)=v(t,z),\ B(t,z)=\ti u(t,z)\textrm{ and }D(t,z)=\ti v(t,z),$$
where
$$\begin{pmatrix}u \\ v
\end{pmatrix}\textrm{ and }\begin{pmatrix}\ti u \\ \ti v
\end{pmatrix}
$$
are Neumann and Dirichlet solutions of \eqref{eqDS}.

We denote by $D(t,\l,z)$ and $R(t,\l,z)$ the numerators of the kernels $K(t,\l,z)$ and $\S(t,\l,z)$ correspondingly:
\begin{equation}
D(t,\l,z)=\det\begin{pmatrix}A(t,z) & \bar A(t,\l) \\ C(t,z) & \bar C (t,\l)
\end{pmatrix}=\frac 1{2i}\det\begin{pmatrix}E(t,z) & E(t,\bar \l) \\ E^\#(t,z) & E^\# (t,\bar \l)
\end{pmatrix},\label{eq1004}\end{equation}
$$
R(t,\l,z)=\det\begin{pmatrix} \cos tz & \overline \cos t\l \\ \sin tz & \overline \sin t\l
\end{pmatrix}=$$$$=\cos tz \sin t\bar \l - \sin tz \cos t\bar \l =\sin [t(\bar \l -z)].
$$

\begin{lemma}\label{lm500}
For a.e. $s\in \R$ and any $C>0$,
\begin{equation}
\sup_{\l,z\in Q(s,C/t)}\left|D(t,\l,z)-\frac{1}{w(s)}R(t,\l,z)\right|=\label{eqDet0}
\end{equation}
$$=\sup_{\l,z\in Q(s,C/t)}\left|D(t,\l,z)-\frac{1}{w(s)}\sin [t(\bar \l -z)]\right|=o(1)
$$
as $t\to\infty$.
\end{lemma}

\begin{proof}
Note that
$$\left|D(t,\l,z)-\frac{1}{w(s)}R(t,\l,z)\right|=\left|K(t,\l,z)-\frac{1}{w(s)}\S(t, \l,z)\right||\bar \l-z|.
$$
Now the statement follows from Lemma \ref{supK-S} because $|\bar \l-z|\lesssim 1/t$ for $\l,z\in Q(s,C/t)$.

\end{proof}


The last lemma admits the obvious self-improvement: one can allow the size of the
box $Q$ tend to zero slower than $1/t$.

\begin{corollary}\label{cor01}
For a.e. $s\in \R$ there exists a function $C(t)>0, C(t)\to\infty$ as $t\to\infty$, such that
\begin{equation}
\sup_{\l,z\in Q(s,C(t)/t)}\left|D(t,\l,z)-\frac{1}{w(s)}R(t,\l,z)\right|=\label{eqDet}
\end{equation}
$$=\sup_{\l,z\in Q(s,C(t)/t)}\left|D(t,\l,z)-\frac{1}{w(s)}\sin [t(\bar \l -z)]\right|=o(1)
$$
as $t\to\infty$.

\end{corollary}

We can now proceed to the approximation of the Hermite-Biehler function $E$ near $s$ in the case when $s$ is close to a resonance (recall that a resonance of the system \eqref{eqDS} is a zero of $E$ or, equivalently, a pole of $\theta_E$).

We define $T_0(s,C)\subset \R_+$ as the set of all $t$ for which
$Q(s,C/t)$ contains a zero  of $E(t,\cdot)$. Here $C$ can be a constant or a function of $t$.

Let us denote by $\g(p)$ the function 

\begin{equation}\g(p)=\sqrt{2}/\sqrt{\sinh [2p]}.\label{eqGamma}\end{equation}

\begin{lemma}\label{sinus} For a.e. $s\in \R$ such that $T_0(s,D)$ is unbounded for some constant $D>1$
there exists $C(t)>D, C(t)\to\infty$ as $t\to\infty$, with the following properties.

 Consider a continuous function $z(t)=x(t)-iy(t)$ on $T_0(s,C)$ such that for each $t\in T_0(s,C)$, $z(t)$ is one of the zeros of $E(t,\cdot) $  in $Q(s,C(t)/t)$. Then for, those $t$ for which $ty(t)>1$,
\begin{equation}\sup_{z\in Q(s,C(t)/t)} \left|E(t,z)- \frac {\a(s,t)\g(ty(t))}{\sqrt{w(s)}}\sin [t(z - z(t))]\right|=o(1),\label{eq1701}\end{equation}
 for some unimodular continuous function $\a(s,t)$ as $t\to\infty,\ t\in T_0(s,C)\cap \{ty(t)>1\}$.
\end{lemma}

\begin{remark}\label{rem007}
As follows from our proof below, if one omits the restriction  $ty(t)>1$ then \eqref{eq1701} holds with
$o(\g(ty(t)))$ instead of $o(1)$ in the right hand side.

The restriction is included in the statement because in the main proofs we
only need the estimates in the case  $ty(t)>1$.

\end{remark}

\begin{proof}
Let $s$ and $C_1(t)$ be such that \eqref{eqDet} is satisfied (with $C=C_1$).

Since $E=A-iC$ vanishes at $z(t)$,
$$A(t,z(t))=iC(t,z(t))=\b$$ and $$D(t,z(t),w)=\b C(t,\bar w)+i\b A(t,\bar w)=i\b E(t,\bar w).$$
Hence, $i\b E(t, w)$ satisfies
$$\sup_{w\in Q(t,C_1/t)}\left|i\b E( t,w)-\frac{1}{w(s)}\sin [t(w-z(t))]\right|=o(1)$$   by Corollary \ref{cor01}.

Let us first establish \eqref{eq1701} for $Q(s,L/t)$ with a constant $L>2D$ in place of $C(t)$.
On one hand, from the last equation for all  $z,w\in Q(s,L/t)$,
$$\det\begin{pmatrix}i\b E(t,z) & i\b E(t,\bar w) \\ -i\bar\b E^\#(t,z) & -i\bar\b E^\# (t,\bar w)
\end{pmatrix}=
$$
$$\frac 1{w(s)^2}\det\begin{pmatrix}  \sin [t(z-z(t))] &  \sin [t(\bar w-z(t))] \\  \sin [t(z-\bar z(t))] &  \sin [t(\bar w-\bar z(t))]\end{pmatrix}+o(1)\psi_1(t,z,w)=
$$
$$ \frac 1{2w(s)^2} [\cos [t((z-\bar w)-2iy(t))] -\cos [t((z-\bar w)]+2iy(t))]+$$
\begin{equation} +o(1)\psi_1(t,z,w)=\frac 1{w(s)^2}\sin [2ity(t)] \sin [t(\bar w -z)]+o(1)\psi_1(t,z,w),\label{eq00a}\end{equation}
as $t\to\infty$ for some bounded function $\psi_1$. On the other hand,  by \eqref{eq1004} and Lemma \ref{lm500}, for  a.e. s, any $z,w\in Q(s,L/t)$
and $t\in T_0(s,L)$,
$$\frac{2i}{w(s)}\sin [t(\bar w -z)]=\det\begin{pmatrix} E(t,z) &  E(t,\bar w) \\  E^\#(t,z) &  E^\# (t,\bar w)
\end{pmatrix}+o(1)\psi_2(t,z,w)
$$
as $t\to\infty$ for some bounded function $\psi_2$.

Comparing the last two equations we obtain
$$2|\b|^2w(s)=(1+o(1))\sinh [2ty(t)]$$
as $t\to\infty,\ t\in T_0(s,L)$.

Altogether, using that $ty(t)>1$, we get
$$\sup_{z\in Q(s,L/t)}\left|E(t,z)-\frac {\a\sqrt{ 2}}{\sqrt{w(s)\sinh [2ty(t)]}}\sin [t(z-z(t))]\right|=o(1)
$$
 for some continuous unimodular  $\a$ as $t\to\infty,\ t\in T_0(s,L)\cap \{ty(t)>1\}$.

Once again, considering larger $L$ the statement can be improved from constant $L$ to $L(t)\to\infty$.
The function $C$ from the statement can be chosen as $C(t)=L(t)$.

\end{proof}

Let $Q_\pm(s,C)=Q(s,C)\cap \C_{\pm}$. In terms of Dirac inner functions $\theta(t,z)$ the last Lemma can
be reformulated as follows

\begin{corollary}\label{cortheta}
If $s, C,\a$ and $z(t)$ are from \eqref{eq1701} then
\begin{equation}\sup_{z\in Q_+(s,C/t)}\left|\theta(t,z)-\bar\a^2\frac{\sin [t(z-\bar z(t))]}{\sin [t(z-z(t))]}\right|=o(1)
\label{eqTheta}\end{equation}
as $t\to\infty, t\in T_0(s,C)\cap \{ty(t)>1\}$.

\end{corollary}

Consider again the box $Q(s,C/t)$ from Corollary \ref{cor01}.
Recall that we denote by $T_0(s,C)$ the set of $t$ such that $Q(s,C/t)$ contains a zero of $E(t,\cdot)$.
In Lemma \ref{sinus} we used Corollary \ref{cor01} to obtain approximations for $E$  for $t\in T_0(s,C)$
Let us now discuss the case when \eqref{eqDet} holds but $t\not\in T_0(s,C)$, i.e., $Q(s,C/t)$ does not contain a zero
of $E(t,\cdot)$. Our goal is to show that then $E$ can be approximated by an exponential near $s$, see Corollary \ref{cor02} below.

Let $I_t=\R\cap Q(s,C/t)$. For $x,w\in I_t$ one can interpret $D(t,x,w)$ as a scalar product
of two $\R^2$-vectors and write \eqref{eqDet} as
$$D(t,x,w)=(A(t,x),C(t,x))^T\cdot (C(t,w),-A(t,w))=$$
\begin{equation}=\frac 1{w(s)}\sin [t(w-x)]+o(1)\psi(t,x),\label{eqsin1}\end{equation}
as $t\to \infty$ for some uniformly bounded $\psi$. Fix $t$ large enough so that
$o(1)\psi(t,x)<<1/w(s)$ and $C(t)>>2\pi$.
Let us consider  two fixed values of $w$ in $I_t$, $w_1$ and $w_2= w_1+\pi/2t$.
Since the last formula must hold for both $w_{1,2}$ and every $x\in I_t$, we see
that the vector $(A(t,x),C(t,x))^T$ has modulus bounded away from zero on $I_t$ and rotates around
the origin as $x$ runs over $I_t$. Since $|I_t|=C(t)/t>>2\pi/t$ the vector makes at least one full rotation.

Hence there exist points $x_0$ and $x_1$ on $I_t$ such that $E(t,x_0)$ is positive and $E(t,x_1)$ is negative imaginary.
Then  $C(t,x_0)=0$ and $A(t,x_1)=0$. Using \eqref{eqsin1} for $x=x_0, w=x_1$ we see that
$A(t,x_0)= c_1$ and $C(t,x_1)= c_2$ where $c_{1,2}$ are positive constants satisfying $c_1c_2=\frac 1{w(s)}\sin [t(x_0-x_1)]+o(1)$.

Using \eqref{eqDet} with $D(t,z,x_0)$ and $D(t,z,x_1)$ we see that $C(t,z)$ is within $o(1)$ from
$$\frac 1{c_1w(s)}\sin [t(z-x_0)]$$ and $A(t,z)$ from $$\frac 1{c_2w(s)}\cos [t(z-x_2)]$$ on $Q(s,C/t)$, where
$tx_2=tx_1-\pi/2$. Therefore
$E(t,z)$ is within $o(1)$ from $$\frac 1{c_1w(s)}\cos [t(z-x_2)]-\frac i{c_2w(s)}\sin [t(z-x_0)]$$
on $Q(s,C/t)$.

Let $\phi(x)=\arg E(t,x)$ and let $J_t=(s-4\pi/t,s+4\pi/t)$.
Notice that if $t|x_2-x_0|>\d\mod 2\pi$, then
$$\frac{\sup_{J_t} \phi'}{\inf_{J_t }\phi'}>1+\e$$
for some $\e=\e(\d)>0$. Since $C(t)\to\infty$, Lemma \ref{lemInner} implies that for large $t$ there is a zero
of $\theta_E$, in $Q(s,C/t)$, which contradicts our assumption that $t\not\in T_0(s,C)$.
Hence, $t|x_2-x_0|=o(1)\mod 2\pi$.
Similarly, we obtain a contradiction if $|c_1-c_2|>\delta$. Since
$$c_1c_2=\frac 1{w(s)}\sin [t(x_0-x_1)]+o(1)=\frac 1{w(s)}+o(1),$$
$c_{1,2}=1/\sqrt{w(s)}+o(1)$
 and $E(t,z)$ is within $o(1)$ from $$\frac {\b(t)}{\sqrt{w(s)}}e^{-itz}$$
on some $Q(s,C_1/t),\ C_1(t)\to\infty,$ for some $\b(t),\ |\b(t)|=1$.


Let $z(t)=u(t)-ip(t)$. Then for the second function in \eqref{eq1701} we have
$$\frac {\g(tp(t))}{\sqrt{w(s)}}\sin [t(z - z(t))]=$$$$=\frac {\g(tp(t))}{\sqrt{w(s)}}\sin [t((z-u(t))+ip(t))]=
$$$$
\frac {\sqrt{ 2}}{\sqrt{w(s)\sinh [2tp(t)]}}\left[\sin [t(z-u(t))]\cos [itp(t)]+\sin [itp(t)]\cos [t(z-u(t))]\right].
$$
Notice that
$$\frac {\sqrt{ 2}}{\sqrt{\sinh [2tp(t)]}}\cos [itp(t)]\to 1\textrm{ and  }\frac {\sqrt{ 2}}{\sqrt{\sinh [2tp(t)]}}\sin [itp(t)]\ \to\ -i
$$
as $p(t)\to\infty$ and therefore the second function in \eqref{eq1701} tends to $$\frac{-i\a(t)}{\sqrt{w(s)}}e^{itz},$$ which is within $o(1)$ from $E$
on $Q(s,C_1/t)$ if we put $\a=i\b$.

Summarizing the above discussion
we see that \eqref{eq1701} holds not only for $t\in T_0(s,C)$, for which $z(t)$ in \eqref{eq1701} can be chosen as a zero of
$E$ in $Q(s,C/t)$ (or a point close to zero as in Lemma \ref{cos} 
 below), but for all $t$ with some $z(t)$. When $t\not\in T_0(s,C)$,
$z(t)$ in \eqref{eq1701} needs to satisfy $\Im\  z(t)>C/t$, in which case the approximating function is close to an exponential on a smaller box.

As before, $T_0(s,C)$ denotes the set of those $t$ for which $Q(s,C/t)$ contains a zero of $E(t,\cdot)$.
We will denote by $T_1(s,C)$ the set of those $t$ for which $Q(s,C/t)$ does not contain a zero $z(t)$ of $E(t,\cdot)$
satisfying $t\Im\  z(t)\geq -1$ (recall that all zeros of $E(t,\cdot)$ are in $\C_-$). Note that,
those $t$ not contained in  $ T_0(s,C)$ also fall into $T_1(s,C)$.
We obtain the following

\begin{corollary}\label{cor02}

1) For a.e. $s$ there exists $C(t)>0, C(t)\to\infty$ as $t\to\infty$, and  $z(t)=x(t)-iy(t)\in \C_-$ such that
\begin{equation}\sup_{z\in Q(s,C/t)}\left| E(t,z)- \frac {\a\g(ty(t))}{\sqrt{w(s)}}\sin [t(z - z(t))]\right|=o(1)\label{eqE=sin}\end{equation}
for some  $\a=\a(s,t), \  |\a |=1$ as $t\to\infty,\ t\in T_1(s,C)$. For $t\in T_0(s,C)$, $z(t)$ can be chosen as a zero of $E(t,\cdot)$.


2) If \eqref{eqE=sin} holds for some $s\in \R$ and some $C(t)>0, C(t)\to\infty$ as $t\to\infty,\ t\in T_1(s,C)$, then for any constant $D>0$,
\begin{equation}\sup_{z\in Q(s,D/t)}\left| E(t,z)- \frac {-i\a(s,t)}{\sqrt{w(s)}}e^{itz}\right|=o(1)\label{eqE=sin1}\end{equation}
as $t\to\infty$, $t\not\in T_0(s,C)$.
\end{corollary}

\begin{remark}\label{rem008}
Similarly to Remark \ref{rem007}, one can remove the restriction $t\in\{ty(t)>1\}$ in 1) and replace the right-hand side of \eqref{eqE=sin}
with  $o(\g(ty(t))$ for $t\to\infty,\ t\in T_0(s,C)$.
\end{remark}

\section{Joint approximations for $E$ and $\ti E$}\label{secJA}

Approximations obtained for the Neumann family of Hermite-Biehler functions $E(t,z)$ in the
previous section are also valid for the Dirichlet family $\ti E(t,z)$.
It will be more convenient for us to use cosines instead of sines for $\ti E$, which corresponds
to the substitution of $z(t)$ with $z(t)+\pi/2t$ in the last statement.

In this section
we establish relations between the parameters of the two approximating functions.


Recall that  $T_0(s,C)$ denotes the set of those $t$ for which $Q(s,C/t)$ contains a zero of $E(t,\cdot)$
and  $T_1(s,C)$ is the set of those $t$ for which $Q(s,C/t)$ does not contain a zero $z(t)$ of $E(t,\cdot)$
with $\Im\  z(t)\geq -1/t$.

\begin{lemma}\label{lemmacos}
For a.e. $s$ there exists $C(t)>0,\ C(t)\to\infty$ with the following properties.

For every $t\in T_1(s,C)$ there exist  $z(t)=u(t)-ip(t), \ti z(t)=\ti u(t)-ip(t)$ and $\alpha(t)=\a(s,t)$  such that
$p(t)> 0, |\a(t)|=1$ and
 $$\sup_{z\in Q(s,C/t)}\left|E(t,z)-\a(t)\frac{\g(tp(t))}{\sqrt{w(s)}}\sin[t(z-z(t))]\right|=o(1)$$ and
 $$\sup_{z\in Q(s,C/t)}\left|\ti E(t,z)-\a(t)\frac{\g(tp(t))}{\sqrt{\ti w(s)}}\cos[t(z-\ti z(t))]\right|=o(1)$$
 as $t\to\infty$. The function $p(t)$ satisfies $p(t)>C(t)$ for $t\not\in T_0(s,C)$ and
 $p(t)\leq C(t)$ for $t\in T_0(s,C)$; $u(t)$ and $\ti u(t)$ satisfy
\begin{equation}\cos [t(\ti u(t)-u(t))]=\sqrt{w(s)\ti w(s)}.\label{fin12}\end{equation}

%

%

\end{lemma}


\begin{proof}


Applying Corollary \ref{cor02} to $E$ and to $\ti E$ we obtain that for a.e. $s$,
\begin{equation}
\begin{gathered}
\sup_{z\in Q(s,D/t)}\left|E(t,z)-\b(t)\frac{\g(ty(t))}{\sqrt{w(s)}}\sin[(t(z-\xi(t)))]\right|=o(1)\textrm{ and }\\
 \sup_{z\in Q(s,D/t)}\left|\ti E(t,z)-\d(t)\frac{\g(t\ti y(t))}{\sqrt{\ti w(s)}}\cos[(t(z-\ti \xi(t)))]\right|=o(1)
 \label{eqapprox1}\end{gathered}\end{equation}
on $Q(s,D/t)$ for some $\xi(t)=x(t)-iy(t)$, $\ti \xi(t)=\ti x(t)-i\ti y(t)$,  any fixed constant $D>4\pi$ and unimodular $\b(t),\d(t)$.

Let us first assume that $\xi(t)\in Q(s,D/t)$.
If the approximating functions from \eqref{eqapprox1} are plugged into the determinant \eqref{eqDet2i} in place of $E$ and $\ti E$ we get
\begin{equation}\begin{gathered} \frac{\g(ty(t))\g(t\ti y(t))}{\sqrt{w(s)\ti w(s)}}(   \b\bar\d \sin [t(z - \xi(t))]\cos [t(z - \bar{ \ti \xi}(t))] - \\
 \bar \b\d \sin [t(z -\bar  \xi(t))]\cos [t(z -  \ti \xi(t))]).
\label{det1}\end{gathered}\end{equation}
 Notice that \eqref{eqapprox1} implies that
$|E|,|\ti E|$ are bounded on $Q(s,D/t)$ uniformly with respect to $t$,
which implies that the  expression in \eqref{det1} is within $o(1)$ from
the determinant in \eqref{eqDet2i} on $Q(s,D/t)$.

For $z=x\in\R$, \eqref{det1} becomes
$$ = 2i \Im\  \left(\frac{\g(ty(t))\g(t\ti y(t))}{\sqrt{w(s)\ti w(s)}} \b\bar\d \sin [t(x - \xi(t))]\cos [t(x - \bar {\ti \xi}(t))]\right)=
$$$$
=2i \Im\  \left(\frac{\g(ty(t))\g(t\ti y(t))}{\sqrt{w(s)\ti w(s)}} \b\bar\d \frac 12\left[ \sin [t(\xi(t)-\bar {\ti \xi}(t))] + \sin [t(2x- (\xi(t) +\bar {\ti \xi}(t)))]    \right]\right).
$$
 By our assumption $ty(t)\leq D$. Suppose that  $t|\ti y(t)-y(t)|>\Delta>0$.
   Then, with the first sine being constant,
the second sine has a comparable absolute value  and its argument grows by more than $2\pi$ on $I_t=Q(s,D/t)\cap\R$. Hence the expression
cannot be within $o(1)$ from $2i$ on $I_t$. This shows that $t|\ti y(t) - y(t)|=o(1)$. Hence we can replace $\ti y(t)$ with $ y(t)$
so that \eqref{eqapprox1} still holds (with a different $o(\cdot)$) and put $p(t)=y(t)$.

In the case $\xi(t)\in Q(s,D/t),\ ty(t)>1$, setting $z= \xi(t), \bar\xi(t)$, we obtain the following  equations from \eqref{det1}:
$$\frac {2i\bar\b\d}{\sqrt{w(s)\ti w(s)}}\cos [t( \ti x(t)-x(t))]=2i + o(1)\textrm{ and }
$$
$$\frac {2i\b\bar\d}{\sqrt{w(s)\ti w(s)}}\cos [t( \ti x(t)-x(t))]=2i +o(1).
$$

From these equations we see that the unimodular constants must satisfy $\b\bar\d=1+o(1)$
and put
$\a(t)=\b(t)=\d(t)+o(1)$.

Calculating the absolute values on each side of either
of the equations we obtain
$$\cos [t(\ti x(t)-x(t))]=\sqrt{w(s)\ti w(s)}+o(1)
$$
to see that $\ti x(t)$ and $x(t)$ can be changed into $\ti u(t)$ and $u(t)$ respectively to satisfy the equations of the Lemma.

We obtain the statement with a constant $D$ in place of $C(t)$. Since $D$ can be chosen arbitrarily large, standard argument allows us
to improve the statement to $C(t)\to\infty$.

In the case $t\not \in  T_0(s,C(t))$, similarly to the previous part one can show that $t\ti y(t)\to\infty$ as well. Then the functions $E$ and $\ti E$ can be approximated by exponential
functions as in \eqref{eqE=sin1}. It follows from the determinant relation \eqref{eqDet2i} that
$\a\bar{\ti\a}=ie^{i\phi}+o(1)$, where $\cos \phi=\sqrt{w(s)\ti w(s)}$. Approximating the exponentials
by $\sin$ and $\cos$ with $p(t)\to\infty$ we obtain the statement.
\end{proof}

\begin{remark}\label{remAC}
It is well known (Alexandrov-Clark formulas) that for a certain bounded analytic function $\phi$ in $\C_+$, $||\phi||_{H^\infty}\leq 1$,
$$w(s)=\frac{1-|\phi(s)|^2}{|1-\phi(s)|^2}, \ \ti w(s)=\frac{1-|\phi(s)|^2}{|1+\phi(s)|^2}\textrm{ and }\sqrt{w(s)\ti w(s)}=\frac{1-|\phi(s)|^2}{|1-\phi^2(s)|}\leq 1.$$
One can see that there always exists  $\d$ such that $\cos\d = \sqrt{w(s)\ti w(s)}$.
In the last statement $u_t$ and $\ti u_t=u_t+\d/t$ satisfy $$\cos [t(\ti u_t-u_t)]=\cos\d=\sqrt{w(s)\ti w(s)}.$$
Because of periodicity one can assume that $|\d|\leq \frac \pi 2$.
Recall that for almost all $s\in\R$, $\sqrt{w(s)\ti w(s)}\neq 0$ and $\d$ can be chosen so that $|\d|< \frac \pi 2$.
\end{remark}

In Lemma \ref{sinus} the  approximation of $E$  by a sine was constructed so that one of the zeros
of $E$  was also a zero of the approximating function.
In some of our future calculations it will be more convenient
for us to choose the approximating functions so that their values at $s$  coincided
with the value of the approximant. Our next Lemma states that it is possible to achieve such an approximation simultaneously for $E$ and $\ti E$ while keeping  the relations
between the parameters of the approximating functions from Lemma \ref{lemmacos}.

\begin{lemma}\label{cos} Suppose that $s\in \R$ and   $C>1$ is a constant. Suppose that  $\sqrt{w(s)\ti w(s)}\neq 0$.

There exists $\e_0>0$ such that for any $\e<\e_0$ the following holds.

If
$$\left|E(t,z) - \frac{\a\g(ty)}{\sqrt{w(s)}}\sin[t(z-(x -iy))]\right|<\e$$ and
$$\left|\ti E(t,z) -\frac{\a\g(ty)}{\sqrt{\ti w(s)}}\cos[t(z-(\ti x -iy))]\right|<\e$$
on $Q(s,3C/t)$ for some $t>0, x, \ti x, y \in \R, \ \a\in \C$, satisfying $ 1<ty<2C, |\a|=1$ and
\begin{equation}-\frac \pi 2< t(\ti x-x) < \frac \pi 2,\ \cos [t(\ti x-x)]=\sqrt{w(s)\ti w(s)},
\label{q-y}
\end{equation}
then
there exist $y', x ', \ti x '\in \R$ and $\a'\in\C, \ |\a'|=1$, satisfying
$$|tx '-tx |+|t\ti x '-t\ti x |+|ty'-ty|+|\a'-\a|<D \e,\ \ti x -x =\ti x '-x ',$$
for some  constant $D=D(C,s)$  and such that
\begin{equation}\begin{gathered}E(t,s) = \a'\frac{\g(ty)}{\sqrt{w(s)}}\sin[t(s-(x' -iy'))]\ \textrm{ and }\\ \ti E(t,s) = \a'\frac{\g(ty)}{\sqrt{\ti w(s)}}\cos[t(s-(\ti x' -iy'))].\label{eq1010}\end{gathered}\end{equation}

\end{lemma}

\begin{proof}
Put $$f(z)=\sqrt{\frac{\ti w(s)}{w(s)}}\cdot \frac{\sin[t(z-(x -iy))]}{\cos[t(z-(\ti x -iy))]}.$$
Note that under the restriction $t|\ti x-x|<\pi/2$, $f$ is not constant. Denote by $J$ the middle third of $I=Q(s,3C/t)\cap \R $.

Note that
$$|E(t,s)/\ti E(t,s)-f(s)|<2\e$$
for all $s\in J$, if $\e_0$ is small enough.
Under the restriction imposed on $y$, $|f'(z)/t|$ and $|f''(z)/t^2|$
are bounded and bounded away from zero in $1/2t$-neighborhood of $J$ for large enough $t$
by  constants depending only on  $\sqrt{w(s)/\ti w(s)}$.
Hence, for  small enough $\e$, there exists $D_1>0$ such that in the disk $B(s,D_1\e/t)$,
$f$ takes all values from $B(f(s),2\e)$. Let $a\in B(s,D_1\e/t)$ be such that
$f(a)=E(s)/\ti E(s)$. Then $\ti x' =\ti x +\Re\  (a-s),\ x' =x +\Re\  (a-s),\ y'=y-\Im\  (a-s)$ will satisfy
$$E(s)/\ti E(s)=\sqrt{\frac{\ti w(s)}{w(s)}}\frac{\sin[t(s+(x' +iy'))]}{\cos[t(s+(\ti x' +iy'))]}.$$
Recalling that $E,\ti E$ satisfy \eqref{eqDet2i}, this implies that
$$|E(s)|=\left|\frac{\g(ty)}{\sqrt{w(s)}}\sin[t(s+(x' +iy'))]\right|$$ and $$|\ti E(s)|=\left|\frac{\g(ty)}{\sqrt{\ti w(s)}}\cos[t(s+(\ti x' +iy'))]\right|
$$
in addition to
$$\arg \frac {E(s)}{\ti E(s)}=\arg  \frac{\sin[t(s+(x' +iy'))] }{ \cos[t(s+(\ti x' +iy'))]}.$$
Hence \eqref{eq1010} will hold with some $\alpha',\ |\a'|=1$. Then  $|\a-\a'|\lesssim\e$ will follow automatically form the inequalities in the statement.

\end{proof}

\section{Proof of the main theorem} \label{secPMT}

Our first goal is to establish that for a. e. $s\in\R$ and  any $C>0$ the box $Q(s,C/t)$ does not contain any resonances of the system for large enough $t$. We will then show that this is equivalent to pointwise convergence of $\nlhat f_T$.

Let $s\in\R$ be such that for some $C>>1$ the box $Q(s,C/t)$ contains a resonance for arbitrarily large $t$. Suppose that $(t_1,t_2)$ is an interval such 
that $Q(s,C/t)\cap\{\Im z<-1/2t\}$ contains a resonance for all $t\in(t_1,t_2)$. Suppose that $s, T$ are such that
the conclusion of Lemma \ref{lemmacos} holds for all $t>T$
with a small $\e$.

{\bf Calculation of the functions $a$ and $b$ for the interval $(t_1,t_2)$. } 
In this section we will find the scattering functions $a_{t_1 \to t_2}$  and $b_{t_1 \to t_2}$ corresponding to an interval $(t_1,t_2)\subset \R_+$.
The shortest way to define such  functions is to say that they are the functions $a, b$, defined as in Section \ref{NLFT}, corresponding
to the system \eqref{eqDS} whose potential function is equal to $f(t+t_1)$ on the interval $(0, t_2-t_1)$ and to $0$ elsewhere.

More constructively, if $M(t,z)$ is the transfer matrix of the system \eqref{eqDS} then one can define the transfer
matrix from $t=t_1$ to $t=t_2$ as
$$M_{t_1 \to t_2}(z)=\begin{pmatrix} A_{t_1 \to t_2}(z) & B_{t_1 \to t_2}(z) \\ C_{t_1 \to t_2}(z) & D_{t_1 \to t_2}(z)\end{pmatrix}
=M(t_2,z)M^{-1}(t_1,z).
$$
After that the Hermite-Biehler functions $E_{t_1 \to t_2}(z)$ and $\ti E_{t_1 \to t_2}(z)$ can be defined as
$$E_{t_1 \to t_2}=A_{t_1 \to t_2}-iC_{t_1 \to t_2},\ \ \ti E_{t_1 \to t_2}=B_{t_1 \to t_2}-iD_{t_1 \to t_2},$$
and
$$a_{t_1 \to t_2}(z)=\frac 12 e^{i(t_2-t_1)z}(E_{t_1 \to t_2}(z)+i \ti E_{t_1 \to t_2}(z)),$$
$$b_{t_1 \to t_2}(z)=\frac 12 e^{i(t_2-t_1)z}(E_{t_1 \to t_2}(z)-i \ti E_{t_1 \to t_2}(z)).$$
Note that $M_{t_1 \to t_2}$ is equal to the transfer matrix $M^*(t_2-t_1,z)$ of the real Dirac system whose potential
function $f^*(t)$ is equal to $f(t+t_1)$ for $0\leq t\leq t_2-t_1$ and to $0$ for $t>t_2-t_1$. Similarly,
the functions $E_{t_1 \to t_2}, \ti E_{t_1 \to t_2}, a_{t_1 \to t_2}$ and $b_{t_1 \to t_2}$ are equal to
the functions $E^*(t_2-t_1,z), \ti E^*(t_2-t_1,z), a^*(t_2-t_1,z)$ and  $b^*(t_2-t_1,z)$ generated by that system.
In particular, Parseval's identity for $a_{t_1 \to t_2}$ becomes
$$||\log |a_{t_1 \to t_2}|\ ||_{L^1(\R)}=|| f||^2_{L^2((t_1,t_2))}.$$

\qquad\qquad

If $t_1>T$, then
by Lemma \ref{cos} 
$$E(s, t_k) = \a_{t_k}\frac{\g(t_ky_{t_k})}{\sqrt{w(s)}}\sin [t_k(s-(x_{t_k}-iy_{t_k}))]$$ and $$\ti E(s, t_k) = \a_{t_k}\frac{\g(t_ky_{t_k})}{\sqrt{\ti w(s)}}\cos [t_k(s-(\ti x_{t_k}-iy_{t_k}))],$$
for $k=1,2$.
Hence
$$M(t_k,s)=\begin{pmatrix} 1/2 & 1/2 \\ -1/2i & 1/2i \end{pmatrix}\begin{pmatrix} E(t_k,s) & \ti E(t_k,s) \\ E^\#(t_k,s) & \ti E^\#(t_k,s) \end{pmatrix}=
\begin{pmatrix} 1/2 & 1/2 \\ -1/2i & 1/2i \end{pmatrix}\times$$$$\footnotesize \times
\begin{pmatrix}\a_{t_k}\frac{\g(t_k y_{t_k})}{\sqrt{w(s)}}\sin [t_k(s-(x_{t_k}-iy_{t_k}))] &
 \a_{t_k}\frac{\g(t_k y_{t_k})}{\sqrt{\ti w(s)}}\cos [t_k(s-(\ti x_{t_k}-iy_{t_k}))]\\
\bar \a_{t_k}\frac{ \g(t_k y_{t_k})}{\sqrt{w(s)}}\sin [t_k(s-(x_{t_k}+iy_{t_k}))] &
\bar \a_{t_k}\frac{ \g(t_k y_{t_k})}{\sqrt{\ti w(s)}}\cos [t_k(s-(\ti x_{t_k}+iy_{t_k}))] \end{pmatrix}
$$
for $k=1,2$.

The transfer matrix $M_{t_1 \to t_2}$ can be calculated as
$$M_{t_1 \to t_2}(z)=M(t_2,z) M(t_1,z)^{-1}=
$$$$
\frac 12 \begin{pmatrix} 1 & 1 \\ i & -i \end{pmatrix}\begin{pmatrix} E(t_2,z) &  \ti E(t_2,z) \\    E^\#(t_2,z) & \ti E^\#(t_2,z)
\end{pmatrix}\begin{pmatrix} E(t_1,z) & \ti E(t_1,z) \\  E^\#(t_1,z) & \ti E^\#(t_1,z) \end{pmatrix}^{-1} 2 \begin{pmatrix} 1 & 1 \\ i & -i \end{pmatrix}^{-1}.
$$
Also
$$
\begin{pmatrix} E_{t_1 \to t_2} & \ti E_{t_1 \to t_2} \\ E^\#_{t_1 \to t_2} & \ti E^\#_{t_1 \to t_2} \end{pmatrix}=2 \begin{pmatrix} 1 & 1 \\ i & -i \end{pmatrix}^{-1}M_{t_1 \to t_2}
$$
\begin{equation}
=\begin{pmatrix} E(t_2,z) &  \ti E(t_2,z) \\    E^\#(t_2,z) & \ti E^\#(t_2,z)\end{pmatrix}\frac 1{2i} \begin{pmatrix} \ti E^\#(t_1,z) & -\ti E(t_1,z) \\  -E^\#(t_1,z) & E(t_1,z) \end{pmatrix}
\begin{pmatrix} 1 & -i \\ 1 & i \end{pmatrix}.
\label{fin10}\end{equation}

After recalling that
$$a_{t_1 \to t_2}=\frac {e^{i(t_2-t_1)z}}2(E_{t_1 \to t_2}+i \ti E_{t_1 \to t_2})$$
we obtain
$$a_{t_1 \to t_2}(z)=$$$$=\frac {e^{i(t_2-t_1)z}}{4i}([(E(t_2,z)\ti E^\#(t_1,z)-\ti E(t_2,z)E^\#(t_1,z)) +$$$$+ (- E(t_2,z)\ti E(t_1,z) + \ti E(t_2,z)E(t_1,z))]
$$$$
+i[-i (E(t_2,z)\ti E^\#(t_1,z)-\ti E(t_2,z)E^\#(t_1,z)) +$$$$+i (- E(t_2,z)\ti E(t_1,z) + \ti E(t_2,z)E(t_1,z))])
$$$$
=\frac {e^{i(t_2-t_1)z}}{2i}( E(t_2,z)\ti E^\#(t_1,z)-\ti E(t_2,z)E^\#(t_1,z) ) .
$$

To shorten our next series of formulas we will use the notations
 $x_k=x_{t_k}-s$, $\ti x_k=\ti x_{t_k}-s$ and $y_k=y_{t_k}$ for $k=1,2$.
 Then $t_k(s-(x_{t_k}-iy_{t_k}))=-t_k(x_k-iy_k)$ and $t_k(s-(\ti x_{t_k}-iy_{t_k}))=-t_k(\ti x_k-iy_k)$. 

The last equation for $a_{t_1 \to t_2}$ leads to

$$a_{t_1 \to t_2}(s)=-\frac {e^{i(t_2-t_1)s}}{2i}( \a_{t_2}\frac{\g(t_2 y_2)}{\sqrt{w(s)}}\sin [t_2(x_2-iy_2)]\bar\a_{t_1}\frac{\g(t_1 y_1)}{\sqrt{\ti w(s)}}\cos [t_1(\ti x_1+iy_1)]- $$$$
\a_{t_2}\frac{\g(t_2 y_2)}{\sqrt{\ti w(s)}}\cos [t_2(\ti x_2-iy_2)] \bar \a_{t_1}\frac{\g(t_1 y_1)}{\sqrt{ w(s)}}\sin [t_1(x_1+iy_1)]    )=
$$$$
=-\frac {e^{i(t_2-t_1)s}}{2i}\frac{\g(t_1 y_1)\g(t_2 y_2)}{\sqrt{w(s)\ti w(s)}}\a_{t_2}\bar\a_{t_1} (\sin [t_2(x_2-iy_2)]\cos [t_1(\ti x_1+iy_1)]-$$$$-\cos [t_2(\ti x_2-iy_2)]\sin [t_1(x_1+iy_1)]   ).
$$
Using several trigonometric identities 
the last expression can be further simplified:
$$
\sin [t_2(x_2-iy_2)]\cos [t_1(\ti x_1+iy_1)]-\cos [t_2(\ti x_2-iy_2)]\sin [t_1(x_1+iy_1)] =
$$$$
=\frac 12(\sin [t_2(x_2-iy_2)+t_1(\ti x_1+iy_1)] +\sin [t_2(x_2-iy_2)-t_1(\ti x_1+iy_1)]-
$$$$
-\sin [t_1(x_1+iy_1)+t_2(\ti x_2-iy_2)]  - \sin [t_1(x_1+iy_1)-t_2(\ti x_2-iy_2)])=
$$$$
=\cos\left[\frac 12 ((t_2(x_2-iy_2)+t_1(\ti x_1+iy_1))+(t_1(x_1+iy_1)+t_2(\ti x_2-iy_2)))\right]\times
$$$$\times\sin\left[\frac 12((t_2(x_2-iy_2)+t_1(\ti x_1+iy_1))-(t_1(x_1+iy_1)+t_2(\ti x_2-iy_2)))\right]+
$$$$
+\cos\left[\frac 12 ((t_2(x_2-iy_2)-t_1(\ti x_1+iy_1)) + (t_1(x_1+iy_1)-t_2(\ti x_2-iy_2)))\right] \times
$$$$\times\sin \left[\frac 12((t_2(x_2-iy_2)-t_1(\ti x_1+iy_1)) - (t_1(x_1+iy_1)-t_2(\ti x_2-iy_2)))\right]=
$$
$$
=\cos\left[\frac 12 (t_2(x_2+\ti x_2)+ t_1(\ti x_1+x_1) ) - i(t_2y_2-t_1y_1)\right]\times
$$$$\times\sin\left[\frac 12(t_2(x_2-\ti x_2) -t_1(x_1 -\ti x_1))\right]+
$$
$$
+\cos\left[\frac 12 (t_2(x_2-\ti x_2) + t_1(x_1-\ti x_1) )\right] \times
$$$$\times\sin \left[\frac 12(t_2(x_2+\ti x_2)-t_1(\ti x_1+x_1)) - i  (t_2y_2+t_1y_1)\right]
$$
 Recall that by \eqref{q-y}, $\cos [t(x_t-\ti x_t)]=\sqrt{w(s)\ti w(s)}$.
 Since $x_t$ changes continuously with $t$, it follows that $t_2(x_2-\ti x_2) =t_1(x_1 -\ti x_1)$ and the last expression is equal to
$$ \sqrt{w(s)\ti w(s)}\sin \left[\frac 12(t_2(x_2+\ti x_2)-t_1(\ti x_1+x_1)) - i  (t_2y_2+t_1y_1)\right].
$$
Altogether we obtain
$$a_{t_1 \to t_2}(s)=$$$$=-\frac {e^{i(t_2-t_1)s}\g(t_1 y_1)\g(t_2 y_2)}{2i}\a_{t_2}\bar\a_{t_1}\sin \left[\frac 12(t_2(x_2+\ti x_2)-t_1(\ti x_1+x_1)) - i  (t_2y_2+t_1y_1)\right]=
$$$$=
\frac { ie^{i(t_2-t_1)s}\a_{t_2}\bar\a_{t_1}}{\sqrt{|\sin [2it_1y_1]\sin [2it_2y_2]|}}\left(\sin \left[\frac 12(t_2(x_2+\ti x_2)-t_1(\ti x_1+x_1)) - i  (t_2y_2+t_1y_1)\right]\right)=
$$$$
=\frac { ie^{i(t_2-t_1)s}\a_{t_2}\bar\a_{t_1}}{\sqrt{|\sin [2it_1y_1]\sin [2it_2y_2]|}}\sin [(t_2x_2-t_1x_1) - i  (t_2y_2+t_1y_1)].
$$

If we put $t_1x_1=u, t_2x_2=u+\e_1, t_1y_1=v, t_2y_2=v+\e_2$ then the last equation
becomes
\begin{equation}
a_{t_1 \to t_2}(s)= \frac {i e^{i(t_2-t_1)s}\a_{t_2}\bar\a_{t_1}}{\sqrt{|\sin [2iv]\sin [2i(v+\e_2)]|}}\sin[(\e_1 -i\e_2)  - 2iv] .
\label{eqAAA}
\end{equation}

Similarly, from \eqref{fin10},
$$b_{t_1 \to t_2}=\frac {e^{i(t_2-t_1)z}}2(E_{t_1 \to t_2}+i \ti E_{t_1 \to t_2})=$$
$$=\frac {e^{i(t_2-t_1)z}}{4i}([(E(t_2,z)\ti E^\#(t_1,z)-\ti E(t_2,z)E^\#(t_1,z)) +$$$$+ (- E(t_2,z)\ti E(t_1,z) + \ti E(t_2,z)E(t_1,z))]
$$$$
-i[-i (E(t_2,z)\ti E^\#(t_1,z)-\ti E(t_2,z)E^\#(t_1,z)) +$$$$+i (- E(t_2,z)\ti E(t_1,z) + \ti E(t_2,z)E(t_1,z))])
$$
\begin{equation}\frac {e^{i(t_2-t_1)z}}{2i}( \ti E(t_2,z)E^\#(t_1,z)-E(t_2,z)\ti E^(t_1,z)) .
\label{fin11}\end{equation}
Next, notice that 
$$\ti E(t_2,z)E^\#(t_1,z)-E(t_2,z)\ti E^(t_1,z)=$$$$
=\a_{t_2}\frac{\g(t_2 y_2)}{\sqrt{\ti w(s)}}\cos [t_2(\ti x_2-iy_2)]  \a_{t_1}\frac{\g(t_1 y_1)}{\sqrt{ w(s)}}\sin [t_1(x_1-iy_1)] -$$ 
$$  - \a_{t_2}\frac{\g(t_2 y_2)}{\sqrt{w(s)}}\sin [t_2(x_2-iy_2)]\a_{t_1}\frac{\g(t_1 y_1)}{\sqrt{\ti w(s)}}\cos [t_1(\ti x_1-iy_1)]  $$
and, using \eqref{fin12} at the end, 
$$\cos [t_2(\ti x_2-iy_2)]  \a_{t_1}\sin [t_1(x_1-iy_1)] -\sin [t_2(x_2-iy_2)]\cos [t_1(\ti x_1-iy_1)=$$
$$=\frac 12 ( \sin  (t_2\ti x_2+t_1 x_1-i(t_2y_2+t_1y_1)) +\sin (t_1 x_1-t_2\ti x_2-i(t_1y_1-t_2y_2))  -
$$
$$-    \sin  (t_2 x_2+t_1 \ti x_1-i(t_2y_2+t_1y_1)) -\sin (t_2 x_2-t_1 \ti x_1-i(t_2y_2-t_1y_1))  )=$$
$$=
\frac 12( \sin (t_1 x_1-t_2\ti x_2-i(t_1y_1-t_2y_2)) -\sin (t_2 x_2-t_1 \ti x_1-i(t_2y_2-t_1y_1)) )=
$$
$$=
\frac 12(\sin (t_1(  x_1-\ti x_1)-(\e_1-i\e_2)) -\sin (t_1(  x_1-\ti x_1)+(\e_1-i\e_2)))=
$$
$$=\sin (\e_1-i\e_2),$$
which gives 
\begin{equation}b_{t_1 \to t_2}=\frac {i e^{i(t_2-t_1)s}\a_{t_2}\a_{t_1}}{\sqrt{|\sin [2iv]\sin [2i(v+\e_2)]|}}\sin (\e_1-i\e_2).
\label{eqBBB}\end{equation}

For 
$$\psi(t_1,t_2,s) =\left( f|_{[t_1,t_2]}   \right) \nlhat{\frac{}{} }(s)$$
we obtain
$$\psi(t_1,t_2,s)=\frac{b_{t_1 \to t_2}(s)}{a_{t_1 \to t_2}(s)}=\a_{t_1}^2\frac{\sin[\e_1   -i\e_2]}{\sin[\e_1   - 2iv-i\e_2] }.$$



$$$$

{\bf Resonances in the box $Q(s,C/t)$.} Our next statement shows that a resonance can only stay near $s$ for a short interval of $t$.

\begin{lemma}\label{corBox} For a. e. $s\in \R$, for any $C>0$, if $I_n(s)$ is the longest sub-interval of $[2^n,2^{n+1}]$  such that $Q(C/t,s)$ contains a zero of $E$ for all
	$t\in I_n(s)$, then $|I_n(s)|/2^n\to 0$ as $n\to\infty$. 
\end{lemma}

\begin{proof} Let $\Sigma, |\Sigma|>0,$ be a set of $s$ such that $I_n(s)$ satisfy
	$$\limsup_{n\to\infty}|I_n(s)|/2^n>2/m$$
	for some $m\in\N$. Let $\Sigma_n$ be the set of $s$ for which $|I_n(s)|>2^n/m$. Then for a subset $S_n$ of $\Sigma_n$ of the size
	$\asymp |\Sigma_n|/m$ the middle thirds of the  intervals $I_n(s)$ have a common point $t_1$ for all $s\in S_n$. Put $t_2=t_1+2^n/3m$.
	
	Using \eqref{eqAAA} and recalling that 
	$t_2(x_{t_2}-s)-t_1(x_{t_1}-s)=\e_1,  t_2y_{t_2}-t_1y_{t_1}=\e_2$ we get that for all $p\in J(s)=\R\cap Q(s,C/2^{n+1})$,
	$$
	|a_{t_1 \to t_2}(p)|\gtrsim |\sin[(\e_1 -i\e_2)  - 2iv+(t_2-t_1)(s-p)]|+o(1)= $$$$= |\sin[(\e_1 -i\e_2)  - 2iv+\frac{2^n(s-p)}{3m}]|+o(1).$$ 
	As $p$ ranges over $J(s), |J(s)|=C/2^{n+1}$, the last function differs from any constant by $\gtrsim 1/m$.
	It follows that 
	$$\int_{J(s)}\log|a_{t_1 \to t_2}(p)|dp\gtrsim |J(s)|/m$$
	when $o(1)$ is small enough. 
	
	Choosing a cover of index two of $S_n$ with the intervals $J(s), s\in S_n$, we conclude that 
	$$|| \log|a_{t_1 \to t_2} ||_1\gtrsim |S_n|.$$
	By the non-linear Parseval equation,
	$$|S_n|\lesssim ||f||^2_{L^2((t_1,t_2))},$$
	which yields
	$$|\Sigma_n|\lesssim ||f||^2_{L^2((2^n,2^{n+1}))}.$$
	Since $\Sigma\subset \cup_{n>N}\Sigma_n$ for all $N>0$, it follows that $|\Sigma|=0$.
	\end{proof}



{\bf Potential modification.} To obtain a contradiction in the next step of the proof we will use the following operation on the potential function $f$. 

Let us consider the modified potential function $f^*$ defined as 
$$f^*(t)=\begin{cases}
	f(t)\text{ for } t<t_1\\
	0\text{ for } t_1\leq t< t_1+\beta\\
	f(t-\beta)\text{ for } t\geq t_1+\beta
\end{cases}
$$
for some positive $\beta$ to be specified later. I.e., $f^*$ is obtained from $f$ by inserting a zero interval of length $\b$ after $t_1$. The functions $E^*,\ti E^*, a^*, b^*$, etc., will now correspond to the system
with the new potential $f^*$.

 Let $\Delta>0$ be a small but fixed constant. 
We assume that  the zero of $E$ moves between $t_1$ and $t_2$ by hyperbolic distance comparable to  $\Delta$. 
Such an assumption can be made because, by Lemma \ref{corBox}, each time a resonance visits the box $Q(s,C/t)$ it must later exit the box $Q(s,2C/t)$, and therefore move by a large hyperbolic distance.

If $t_2>t_1>T$ are large enough we can assume that $o(1)$ in the conclusion of Lemma \ref{lemmacos} is $<<\Delta^4$.
Note that Lemma \ref{corBox} implies that
for sufficiently large $T$,  $\frac{t_2-t_1}{t_1}<<\Delta^3$.

In what follows we will write $\approx$ if the two quantities are equal up to an error of the size $\lesssim \Delta^4+\frac{t_2-t_1}{t_1}<<\Delta^3$. 


%

Next, we will show that  $E^*_{0\to t_2+\beta}(s), \ti E^*_{0\to t_2+\beta}(s)$ can be approximated similar to Lemma \ref{lemmacos}. 
Using again the transfer matrix multiplication together with the relation $E^*_{t_1\to t_1+\b}(s)=e^{-i\b s}, \ti E^*_{t_1\to t_1+\b}(s)=-ie^{-i\b s}$, we obtain

$$\begin{pmatrix} E^*(t_2+\beta,s) & \ti E^*(t_2+\beta,s) \\ (E^*)^\#(t_2+\beta,s) &( \ti E^*)^\#(t_2+\beta,s) \end{pmatrix}=\begin{pmatrix} E_{t_1\to t_2}(s) & \ti E_{t_1\to t_2}(s) \\ E^\#_{t_1\to t_2}(s)& E^\#_{t_1\to t_2}(s) \end{pmatrix}\begin{pmatrix} 1/2 & 1/2 \\ -1/2i & 1/2i \end{pmatrix}\times$$$$
\times\begin{pmatrix} e^{-i\beta s} & -ie^{-i\beta s} \\ e^{i\beta s}  & ie^{i\beta s} \end{pmatrix}\begin{pmatrix} 1/2 & 1/2 \\ -1/2i & 1/2i \end{pmatrix}\begin{pmatrix} E(t_1,s) & \ti E(t_1,s) \\ E^\#(t_1,s) & \ti E^\#(t_1,s) \end{pmatrix}\approx
$$\begin{equation}
\begin{pmatrix} c & d \\ q & r \end{pmatrix}\begin{pmatrix} \frac 1{\sqrt{w(s)}}\sin (t_1(s-z(t_1))) &  \frac 1{\sqrt{\ti w(s)}}\cos (t_1(s-z(t_1)+p))  \\  \frac 1{\sqrt{w(s)}}\sin (t_1(s-\bar z(t_1))) & \frac 1{\sqrt{\ti w(s)}} \cos (t_1(s-\bar z(t_1)+p))  \end{pmatrix}
\label{eqMat}\end{equation}
for some complex constants $c,d,q,r$ and a real constant $p=\ti x_1-x_1,\cos p =\sqrt{w(s)\ti w(s)}$.

Hence, using trigonometric identities,
$$\sqrt{w(s)}E^*(t_2+\beta,s)\approx c \sin (t_1(s-z(t_1))) +d\sin (t_1(s-\bar z(t_1)))\approx$$$$
\approx c \sin (t_2(s-z(t_1))) +d\sin (t_2(s-\bar z(t_1)))\approx$$$$
\approx c \sin ((t_2+\b)(s-z(t_1))) +d\sin ((t_2+\b)(s-\bar z(t_1))).$$
Here we use that $|s-z(t_1)|<2C/t_1$ and $(t_2-t_1)/t_1, \b/t_1$ are small.
Our chain of trigonometric identities can be further continued:
$$\sqrt{w(s)}E^*(t_2+\beta,s)\approx c \sin ((t_2+\b)(s-z(t_1))) +d\sin ((t_2+\b)(s-\bar z(t_1)))=$$$$
= A\sin((t_2+\b)s)+B\cos((t_2+\b)s),$$
where 

\begin{equation}\begin{gathered}A=c\cos ((t_2+\b)z(t_1))+d\cos ((t_2+\b)\bar z(t_1)),\\ B= -d\sin ((t_2+\b)\bar z(t_1))-c\sin ((t_2+\b)z(t_1)).\end{gathered}\label{eqAB000}\end{equation}

Notice that $A^2+B^2\neq 0$ if $\Delta << v$. Indeed, in this case \eqref{eqAAA} and \eqref{eqBBB} show that $a_{t_1\to t_2}$ is close (within $\asymp \DD/v$) to $1$ and $b_{t_1\to t_2}$ is close to 0.  Hence, by \eqref{eqab}, 
$$\begin{pmatrix} E_{t_1\to t_2}(s) & \ti E_{t_1\to t_2}(s) \\ E^\#_{t_1\to t_2}(s)&\ti  E^\#_{t_1\to t_2}(s) \end{pmatrix}
$$
is close to 
$$\begin{pmatrix} e^{-i(t_2-t_1)s} & -ie^{-i(t_2-t_1)s} \\ e^{i(t_2-t_1)s}& ie^{i(t_2-t_1)s}\end{pmatrix}.$$
it follows that in \eqref{eqMat} the matrix
$$\begin{pmatrix} c & d \\ q & r \end{pmatrix}$$
is close to 
$$\gamma(v)\begin{pmatrix} e^{-i((t_2-t_1)s+\b)} & 0 \\ 0& ie^{i((t_2-t_1)s+\b)}\end{pmatrix},$$
i.e., $|c|$ is close to $\gamma(v)$ and $|d|$ is close to zero. Now we can see that $|A^2+B^2|$ given by \eqref{eqAB000} is close to $\g^2(v)\neq 0$. 

By choosing a complex constant $C$ such that $C^2=A^2+B^2$ and choosing the complex number $z^*(t_2+\beta)$ to be the closest number to $z(t_1)$ with the property that 
$$\sin [z^*(t_2+\beta)]=\frac A{C}\text{ and } \cos[ z^*(t_2+\beta)]=\frac B{C},$$
we obtain 
$$\sqrt{w(s)}E^*(t_2+\beta,s)\approx 
A\sin((t_2+\b)s)+B\cos((t_2+\b)s)=C\sin((t_2+\b)(s-z^*(t_2+\beta))).$$

Similarly,
$$\sqrt{\ti w(s)}\ti E^*(t_2+\beta,s)\approx a \cos (t_1(s-z(t_1)+p)) +b\cos (t_1(s-\bar z(t_1)+p))\approx$$
$$\approx a \sin ((t_2+\beta)(s-z(t_1)+p)-\pi/2) +b\sin ((t_2+\beta)(s-\bar z(t_1)+p)-\pi/2)\approx$$
$$\approx A\sin((t_2+\beta)(s+p)-\pi/2)+B\cos((t_2+\beta)(s+p)-\pi/2)\approx$$$$\approx C\sin((t_2+\b)(s-z^*(t_2+\beta)+p)-\pi/2)=$$
$$=C\cos((t_2+\b)(s-z^*(t_2+\beta)+p)).$$

Using the relation \eqref{eqDet2i}, similar to the argument from the proof of Lemma \ref{sinus}, one can show that $C$ can be chosen as 
$$C=\alpha^*(t_2+\b)\g((t_2+\b)|\Im z^*(t_2+\b)|)
$$
for some unimodular constant $\alpha^*(t_2+\b)$.

Let $$\e=t_2(s-z(t_2))-t_1(s-z(t_1)),$$$$
  \e^*=(t_2+\b)(s-z^*(t_2+\b))-(t_1+\b)(s-z^*(t_1+\b))$$ 
 (where $z^*(t_1+\b)=z(t_1)$, as follows from \eqref{eqDynRes} since $f^*=0$ on $(t_1,t_1+\b)$).

By our construction $\psi(t_1,t_2,s)$ should coincide with $\psi^*(t_1+\beta,t_2+\beta,s)$, i.e., 

$$\a_{t_1}^2\frac{\sin[\e_1   -i\e_2]}{\sin[\e_1 -i\e_2  - 2iv] }=(\a^*_{t_1+\beta })^2\frac{\sin[\e^*_1   -i\e^*_2]}{\sin[\e^*_1  -i\e^*_2 - 2iv^*] }\approx$$\begin{equation}\approx (\a^*_{t_1+\beta })^2\frac{\sin[\e^*_1   -i\e^*_2]}{\sin[\e^*_1  -i\e^*_2 - 2iv] }
\label{eq11A}\end{equation}
since $v\approx v^*$.

Note that the choice of the constant $\beta$ in the definition of $f^*$ affects the value of $\a^*_{t_1+\beta }$: 
$$(\a^*_{t_1+\beta })^2\approx e^{-2i\beta s}\a_{t_1}^2,$$ 
as can be seen, for instance, from \eqref{eqArgE}. 


Then   \eqref{eq11A}  implies
\begin{equation}\frac{\sin[\e]}{\sin[\e   - 2iv] }\approx e^{-2i\beta s}\frac{\sin[\e^{*}]}{\sin[\e^{*}   - 2iv] }.
	\label{eq11B}\end{equation}
	
	This relation implies in particular that by changing $\b$ in the definition of $f^*$ we can change the argument of $\e^*$, which will be used in our constructions below.

{\bf Special case of resonance motion.}
First assume that for some $t_1<t_2<t_3$ the resonance moves in the following way: 

$\bullet$ for the interval $(t_1,t_2)$ and $\e=\e_1-i\e_2$ defined as above, we have $\e=i\DD$;

$\bullet$ for the interval $(t_2,t_3)$ we have $\tau_1=t_3(s-x_{t_3})-t_2(s-x_{t_2})=0$ and $\tau_2=t_3y_{t_3}-t_2y_{t_2}=\DD$ so that for the motion $\tau$ of the resonance on $(t_2,t_3)$
we again have that $\tau=\tau_1-i\tau_2=i\DD$.  

Let us now insert a zero  interval in the definition of $f^*$ as described above and choose $\beta\approx\pi/4$ so that 
$\a^*_{t_1+\b}=-\a_{t_1}$ and 
\eqref{eq11A} gives
$$\frac{\sin[\e]}{\sin[\e   - 2iv] }=\frac{\sin[i\DD]}{\sin[i(\DD   - 2v)] }=-\frac{\sin[\e^{*}]}{\sin[\e^{*}   - 2iv^*] }.$$
It follows that $\e^*$ is purely imaginary,  $\e^*= -i\DD_1$ for some $\DD_1>0$.

Next, notice that for $\tau^*=\tau^*_1-i\tau^*_2$, where 
$$\tau^*_1=t_3(s-x^*_{t_3+\b})-t_2(s-x^*_{t_2+b})\text{ and }\tau^*_2=t_3y^*_{t_3+\b}-t_2y^*_{t_2+\b},$$
by applying the same argument to the intervals $(t_1,t_3)$ and $(t_1+\b,t_3+\b)$ we will  have 
$$\frac{\sin[2i\DD]}{\sin[i(2\DD   - 2v)] }=-\frac{\sin[\e^*+\tau^{*}]}{\sin[\e^{*} +\tau^*  - 2iv^*] }
=-\frac{\sin[\tau^{*}-i\DD_1]}{\sin[\tau^*   -i(\DD_1+2v^*)] }.$$
It follows again that $\tau^*$ is purely imaginary, $\tau^*=-i\DD_2$ for some $\DD_2>0$. 

If one now applies  \eqref{eq11B} for the intervals $(t_1,t_2)$ and  $(t_1+\b,t_2+\b)$, comparing the absolute values of 
the functions $\psi(t_1,t_2,s)$ and $\psi(t_1+\b,t_2+\b,s)$ we get
\[
\frac{\sinh \Delta_1}{\sinh(v+\Delta_1)}
\approx\frac{\sinh \Delta}{\sinh(v-\Delta)}.
\]
For the intervals $(t_2,t_3)$ and  $(t_2+\b,t_3+\b)$ we obtain 
\[
\frac{\sinh \Delta_2}{\sinh(v+\Delta_1+\Delta_2)}
\approx \frac{\sinh \Delta}{\sinh(v-2\Delta)}.
\]
 Finally, for $(t_1,t_3)$ and  $(t_1+\b,t_3+\b)$,
 	\[
 \frac{\sinh(2\Delta)}{\sinh(v-2\Delta)}
 \approx
 \frac{\sinh(\Delta_1+\Delta_2)}{\sinh(v+\Delta_1+\Delta_2)} .
 \]
 
 However, Lemma \ref{lemAppendix} in the Appendix implies that for small enough $\DD$ the constants $\DD_1$ and $\DD_2$ satisfying the last three relations do not exist.

{\bf General case.} Let now $t_1\leq t_2\leq t_3$ be such that $\DD_1=|\e|$ and $\DD_2=|\tau|$, where as before $\e$ corresponds to the motion of the resonance on $(t_1,t_2)$ and $\tau$ on $(t_2,t_3)$, $\DD_1,\DD_2\asymp \DD$. 

Insert a zero interval $(t_1,t_1+\b_1)$ after $t_1$ to obtain $f^*$, where $\b_1$ is chosen so  that $\e^*$
corresponding to the interval $(t_1,t_2+\b_1)$ is equal to $i\DD_1^*$ for some $\DD_1^*\geq 0$. 

(Note that the existence of such a $\b_1$ is implied by \eqref{eq11B}. Indeed, as $\b$ ranges over an interval of length about $2\pi$,
$\e^*$ satisfying \eqref{eq11B} makes a full turn around the origin.)

After that insert a second zero interval $(t_2+\b_1, t_2+\b_1+\b_2)$ after $t_2+\b_1$ to obtain $f^{**}$ from $f^*$ so that 
for the motion of the resonance on $(t_2+\b_1, t_3+\b_1+\b_2)$ we had $\tau^{**}=i\DD_2^{**}$
 with $\DD_2^{**}\geq 0$. 

Note that by moving $t_2$ between $t_1$ and $t_3$ one can achieve $\DD^*_1=\DD^{**}_2$ in the above construction. Indeed,
if $t_1=t_2$ then $0=\DD_1\approx \DD^*_1<\DD^{**}_2$ and if $t_2=t_3$ then
$\DD_1^*>\DD_2^{**}\approx \DD_2=0$. By continuity this implies that $t_2$ can be chosen between $t_1$ and $t_3$ so that
$\DD^*_1=\DD^{**}_2\asymp \DD$. 

Now, for the modified potential $f^{**}$ the motion of the resonance over the intervals $(t_1, t_2+\b_1)$ and $(t_2+\b_1,t_3+\b_1+\b_2)$, fits the special case considered in the previous part of the proof and Lemma \ref{lemAppendix}, once again, yields a 
contradiction.



{\bf Final step}. Note that our calculations were made under the assumption of the presence of a zero of $E(t_0)$ in the box $Q(C/t, s)$ for a sufficiently large $t$. The obtained contradiction shows that for any $C>0$, for  a. e. $s$  there exists $T=T(s)$ such $Q(C/t, s)$ contains no zeros of $E$ for $t>T$.

Hence, for a.e. $s\in \R$ there exists $C(t)>0,\ C(t)\to\infty$ as $t\to\infty$, such that
$Q(s,C(t)/t)$ has no zero of $E(t,\cdot)$. Corollary \ref{cor02}, part 2), now implies that
for a.e. $s$ and any $D>0$,
\begin{equation}\sup_{Q(s,D/t)}\left|E(t,z)-\frac {-i\a(s,t)}{\sqrt{ w(s)}}e^{itz}\right|=o(1)\label{eq990}\end{equation}
as $t\to\infty$. Since $|\a|=1$, it follows that
$$|E(t,s)|\to \frac 1{\sqrt{w(s)}}$$
as $t\to\infty$ for a.e. $s$.

Similarly, for a.e. $s$,
\begin{equation}\sup_{Q(s,D/t)}\left|\ti E(t,z)-\frac {-i\b\a}{\sqrt{\ti w(s)}}e^{itz}\right|=o(1)\label{eq991}\end{equation}
as $t\to\infty$, where $\a$ is the same as in the previous formula.
The equation \eqref{eqDet2i} implies that $\beta(t)$ can be chosen as a constant
 $\b=e^{-i\phi},\ \phi=\arcsin \sqrt{w(s)\ti w(s)}$.
Therefore
$$|\ti E(t,s)|\to \frac 1{\sqrt{\ti w(s)}}$$
as $t\to\infty$ for a.e. $s$.

Together with \eqref{eqab} this implies
$$|a(t,s)|\to\frac 12\sqrt{\frac 1{w(s)}+\frac 1{\ti w(s)}+2\frac{\cos\left(\frac \pi 2 -\phi\right)}{\sqrt{w(s)\ti w(s)}}}=\frac 12\sqrt{\frac 1{w(s)}+\frac 1{\ti w(s)}+2}$$
at a.e. $s$ and
$$|b(t,s)|\to\frac 12\sqrt{\frac 1{w(s)}+\frac 1{\ti w(s)}-2\frac{\cos\left(\frac \pi 2 -\phi\right)}{\sqrt{w(s)\ti w(s)}}}=\frac 12\sqrt{\frac 1{w(s)}+\frac 1{\ti w(s)}-2}$$
at a.e. $s$ as $t\to\infty$. 

Note that Remark \ref{remAC}, along with the inequality of arithmetic and geometric means, implies that the last expression under the root is non-negative. 
Convergence of $|a|$ a.e. and convergence of $||\log|a| ||_1$, which follows from Parseval's identity, implies
convergence of $\log|a|$ in $L^1$.

Also,
$$\sup_{Q(s,D/t)}\left|\frac ba - \frac{\frac 1{\sqrt{w(s)}}+\frac {i\b}{\sqrt{\ti w(s)}}}{\frac 1{\sqrt{w(s)}}-\frac {i\b}{\sqrt{\ti w(s)}} }   \right|=
o(1)$$
and therefore $\nlhat f=b/a$ converges pointwise a.e. on $\R$.

\section{Appendix}
\begin{lemma}\label{lemAppendix}
	Let $v>0$ be fixed. Let $\Delta$ be  a small positive constant and let 
	$\Delta_1>0$ and $\Delta_2>0$ be such that
	\[
	\left|\frac{\sinh \Delta_1}{\sinh(v+\Delta_1)}
	-\frac{\sinh \Delta}{\sinh(v-\Delta)}\right| <\Delta^3
	\]and \[
	\left|\frac{\sinh \Delta_2}{\sinh(v+\Delta_1+\Delta_2)}
	-\frac{\sinh \Delta}{\sinh(v-2\Delta)}\right| <\Delta^3.
	\]
	Then there exist constants $c=c(v)>0$ and $\Delta_0=\Delta_0(v)>0$ such that
	for all $0<\Delta<\Delta_0$,
	\[
	\frac{\sinh(2\Delta)}{\sinh(v-2\Delta)}
	-
	\frac{\sinh(\Delta_1+\Delta_2)}{\sinh(v+\Delta_1+\Delta_2)}
	\;>\;
	c\,\Delta^2 .
	\]
\end{lemma}

\begin{proof}
	Put
	\[
	F_v(x)=\frac{\sinh x}{\sinh(v+x)}, 
	\qquad s_v=\sinh v,\quad c_v=\cosh v.
	\]
	
	\medskip
	\noindent\textbf{Expansion of $F_v(x)$ at $x=0$.}  
	We use
	\[
	\sinh(v+x)=s_v + c_v x + \tfrac12 s_v x^2 + O(x^3), 
	\qquad
	\sinh x = x + \tfrac{1}{6}x^3 + O(x^5).
	\]
	Thus
	\[
	F_v(x)
	= \frac{x}{s_v} - \frac{c_v}{s_v^2}x^2 + O(x^3).
	\]
	Consequently,
	\[
	\frac{\sinh \Delta}{\sinh(v-\Delta)}
	= \frac{\Delta}{s_v} + \frac{c_v}{s_v^2}\Delta^2 + O(\Delta^3),
	\qquad
	\frac{\sinh (2\Delta)}{\sinh(v-2\Delta)}
	= \frac{2\Delta}{s_v} + \frac{4c_v}{s_v^2}\Delta^2 + O(\Delta^3).
	\]
	
	\medskip
	\noindent\textbf{Expansion of $\Delta_1$.}
	To expand $\DD_1$ in terms of $\DD$, let us write
	\[
	\Delta_1=\Delta+\alpha_1 \Delta^2+O(\Delta^3).
	\]
	Since
	\[
	F_v(\Delta_1)=\frac{\sinh\Delta}{\sinh(v-\Delta)} + O(\Delta^3),
	\]
	matching coefficients of $\Delta^2$ gives
	\[
	\frac{\alpha_1}{s_v}-\frac{c_v}{s_v^2}= \frac{c_v}{s_v^2},
	\]
	so
	\[
	\alpha_1=\frac{2c_v}{s_v},
	\qquad
	\Delta_1=\Delta+\frac{2c_v}{s_v}\Delta^2+O(\Delta^3).
	\]
	
	\medskip
	\noindent\textbf{Expansion of $\Delta_2$.} 
	To obtain an expansion for $\DD_2$ we can write
	\[
	\Delta_2=\Delta+\alpha_2\Delta^2+O(\Delta^3),\qquad
	u:=\Delta_1+\Delta_2.
	\]
	Using the expansion of $\Delta_1$ we have
	\[
	u=2\Delta+(\alpha_1+\alpha_2)\Delta^2+O(\Delta^3).
	\]
	The (approximate) defining equation is
	\[
	\frac{\sinh\Delta_2}{\sinh(v+u)}
	= \frac{\sinh\Delta}{\sinh(v-2\Delta)} + O(\Delta^3).
	\]
	
	The numerator expands as
	\[
	\sinh\Delta_2=\Delta+\alpha_2\Delta^2+O(\Delta^3).
	\]
	The denominator satisfies
	\[
	\sinh(v+u)=s_v + 2c_v\Delta 
	+ \bigl(c_v(\alpha_1+\alpha_2)+2s_v\bigr)\Delta^2 + O(\Delta^3).
	\]
	Using the reciprocal expansion and keeping terms up to $\Delta^2$ gives
	\[
	\frac{\sinh\Delta_2}{\sinh(v+u)}
	= \frac{\Delta}{s_v}
	+ \Bigl(-\frac{2c_v}{s_v^2}+\frac{\alpha_2}{s_v}\Bigr)\Delta^2
	+ O(\Delta^3).
	\]
	The right-hand side expands as
	\[
	\frac{\sinh\Delta}{\sinh(v-2\Delta)}
	= \frac{\Delta}{s_v} + \frac{2c_v}{s_v^2}\Delta^2 + O(\Delta^3).
	\]
	Matching the coefficients of $\Delta^2$ gives
	\[
	-\frac{2c_v}{s_v^2}+\frac{\alpha_2}{s_v}
	= \frac{2c_v}{s_v^2},
	\qquad
	\text{so}\qquad
	\alpha_2=\frac{4c_v}{s_v}\qquad
	\text{and}\qquad
\a_1+	\alpha_2=\frac{6c_v}{s_v}.
	\]
	Thus
	\[
	\Delta_2=\Delta+\frac{4c_v}{s_v}\Delta^2+O(\Delta^3)\quad\text{and}\quad
	u=\Delta_1+\Delta_2=2\Delta+\frac{6c_v}{s_v}\Delta^2+O(\Delta^3).
	\]
	
	\medskip
	\noindent\textbf{Final expansion.}
	Using the expansion of $F_v$,
	\[
	\frac{\sinh(u)}{\sinh(v+u)}
	= \frac{2\Delta}{s_v}+\frac{2c_v}{s_v^2}\Delta^2 + O(\Delta^3).
	\]
	Subtracting from the expansion of 
	\(\dfrac{\sinh(2\Delta)}{\sinh(v-2\Delta)}\) gives
	\[
	\frac{\sinh(2\Delta)}{\sinh(v-2\Delta)}
	-
	\frac{\sinh(\Delta_1+\Delta_2)}{\sinh(v+\Delta_1+\Delta_2)}
	= \frac{2c_v}{s_v^2}\Delta^2 + O(\Delta^3),
	\]
	which implies the statement.
\end{proof}

\end{document}